\documentclass[10pt,reqno]{amsart}
\usepackage{amsmath, amsthm, amssymb, enumerate}
\usepackage[usenames]{color}
\usepackage{eepic,epic}
\usepackage{esint}

\newtheorem{thm}{Theorem}[section]

\newtheorem{lem}[thm]{Lemma}

\theoremstyle{definition}
\newtheorem{defn}[thm]{Definition}
\theoremstyle{remark}

\numberwithin{equation}{section}

\newcommand{\ep}{\epsilon}

\newcommand{\pa}{\partial}

\newcommand{\R}{\mathbb{R}}

\newcommand{\bq}{\begin{equation}}
\newcommand{\eq}{\end{equation}}

\newcommand{\lt}{\left}
\newcommand{\rt}{\right}
\newcommand{\U}{\mathcal{U}}

\newcommand{\ml}{\mathcal{L}}

\begin{document}

\title[A sharp error analysis for the DG method of optimal control problems]
{A sharp error analysis for the discontinuous Galerkin method of optimal control problems}

\subjclass[2000]{65L05, 49J15, 49M25, 65L60}

\date{\today}

\author{Woocheol Choi}
\address[Woocheol Choi]{\newline
Department of Mathematics \newline  Sungkyunkwan University, Suwon 16419, Korea (Republic of)}
\email{choiwc@skku.edu}

\author{Young-Pil Choi}
\address[Young-Pil Choi]{\newline Department of Mathematics \newline Yonsei University, Seoul 03722, Korea (Republic of)}
\email{ypchoi@yonsei.ac.kr}

\begin{abstract}In this paper, we are concerned with a nonlinear optimal control problem of ordinary differential equations. We consider a discretization of the problem with the discontinuous Galerkin method with arbitrary order $r \in \mathbb{N}\cup \{0\}$. Under suitable regularity assumptions on the cost functional and solutions of the state equations, we provide sharp estimates for the error of the approximate solutions. Numerical experiments are presented supporting the theoretical results.
\end{abstract}

\maketitle

\tableofcontents

%
%
%
%

\section{Introduction}

In the present work, we discuss discontinuous Galerkin (DG) approximations to a nonlinear optimal control problem (OCP) of ordinary differential equations (ODEs). More precisely, we consider the following optimal control problem:
\bq\label{main_eq1}
\mbox{Minimize } J(u,x) := \int_0^T g(t,x(t),u(t))\,dt
\eq
subject to 
\begin{align}\label{main_eq2}
\left\{\begin{aligned}
x'(t) &= f(t,x(t),u(t)), \quad \mbox{a.e. on } [0,T], \cr
x(0) &= x_0, \cr
{u} &\in \U_{ad}, \quad \mbox{a.e. on } [0,T].
\end{aligned}\right.
\end{align}
Here $u(t) \in \R^m$ is the control, and $x(t) \in \R^d$ is the state of the system at time $t \in [0,T]$. Further, $g: [0,T] \times \R^d \times \R^m \to \R$ and $f: [0,T] \times \R^d \times \R^m \to \R^d$ are given, and the set of admissible controls $\U_{ad} \subset {\mathcal{U}:=L^{\infty}(0,T; \R^m)}$ is given by 
{
\[
\U_{ad} := \{ u(t) \in \R^m : u_\ell \leq u(t) \leq u_u\}
\]
for some $u_\ell, u_u \in \R^m$. }

There have been a lot of study on the numerical computation for the above problem. The numerical schemes need a discretization of the ODEs, for example, the Euler discretization for the OCPs of ODEs are well studied for sufficiently smooth optimal controls based on strong second-order optimality conditions \cite{Alt84, DH93, DH00}. For optimal control problems with control appearing linearly, the optimal control may be discontinuous, for an instance, bang-bang controller, and such conditions are not satisfied. In that respect, there have been many studies to develop new second-order optimality conditions for the optimal control problems with control appearing linearly \cite{AFS18, Fel03, Osm05, Osm07}.

 The Pseudo-spectral method is also popularly used for the discretization due to its capability of high-order accuracy for smooth solutions to the OCPs \cite{EKR, RK}. However, the high-order accuracy of the Pseudo-spectral method is known to be often lost for bang-bang OCPs, where the solutions may not be smooth enough. To handle this issue, Henriques et al. \cite{HLE} proposed a mesh refinement method based on a high-order DG method for the OCPs of ODEs. The DG method discretizes the time interval in small time subintervals, in which the weak formulation is employed. The test functions are usually taken as piecewise polynomials which can be discontinuous at boundaries of the time interval, see Section \ref{sec-2} for more detailed discussion. We refer to \cite{Bac16, E95, SS00} and references therein for DG methods for ODEs.

In this paper, we provide a rigorous analysis for the DG discretization applied to the nonlinear OCP \eqref{main_eq1}-\eqref{main_eq2} with arbitrary order $r \in \mathbb{N} \cup \{0\}$ for  general functions $f$ and $g$ with suitable smoothness. Motivated from a recent work by Neitzel and Vexler \cite{NV12}, we impose the non-degeneracy condition \eqref{ass_1} on an optimal control $\bar{u}$ of the OCP \eqref{main_eq1}-\eqref{main_eq2}. 
We obtain the existence and convergence results  for the semi-discretized case and the fully discretized case.  The rates of the convergence results depend on the regularity of the optimal solution $\bar{u}$ and  its adjoint state with the degree of piecewise polynomials mentioned above, see Section \ref{sec-2} for details.

It is worth noticing that the control is not required to be linear in the state equations \eqref{main_eq2}, and the control space $\U_{ad}$ allows to take into account discontinuous controls. 
The constraints for controls are defined by lower and upper bounds. Moreover, the cost functional is also given in a general form, it may not be quadratic.

For notational simplicity, we denote by $I:=(0,T)$, $X:= L^2(I;\R^d)$, and $(v,w)_I = (v,w)_{L^2(I;\R^d)}$. We also use simplified notations:
 \[
 \|\cdot\|_{L^p(I)} := \|\cdot\|_{L^p(I;\R^d)} \quad \mbox{and} \quad \|\cdot\|_{W^{p,\infty}(I)} := \|\cdot\|_{W^{p,\infty}(I;\R^d)} 
 \]
 for $1 \leq p \leq \infty$. Throughout this paper, {for any compact set $K \subset \R^m$, we assume that $f, g \in C([0,T]; W^{3,\infty}(\R^d \times K))$ satisfy 
\begin{equation}\label{eq-1-10}
\sup_{0 \leq t \leq T} \lt(\|f(t,\cdot,\cdot)\|_{W^{3,\infty}} + \|g(t,\cdot,\cdot)\|_{W^{3,\infty}}\rt)\leq M
\end{equation}
for some $M > 0$.}

We next introduce the control-to-state mapping $G: \U \to X \cap L^\infty(I;\R^d)$, $G(u) = x$, with $x$ solving \eqref{main_eq2}. It induces the cost functional $j: \U \to \R_+$, $u \mapsto J(u,G(u))$. This makes the optimal control problem \eqref{main_eq1}-\eqref{main_eq2} equivalent to
\bq\label{main_eq3}
\mbox{Minimize $j(u)$ subject to $u \in \U_{ad}$.} 
\eq

\begin{defn}\label{def_ls}A control $\bar u \in \U_{ad}$ is a local solution of \eqref{main_eq3} if there exists a constant $\ep > 0$ such that $j(u) \geq j(\bar u)$ holds for all $u \in \U_{ad}$ with $\|\bar u - u\|_{L^2 (I)} \leq \ep$.
\end{defn}
In the proof of the existence and convergence results,  the main task is to show that the strong convexity of $j$ induced by the second order condition \eqref{ass_1}  is preserved near  the optimal control $\bar{u}$ and also for its DG discretized version $j_h$. It will be achieved using the second-order analysis in Section \ref{sec-4}. As a preliminary, we also justify that $j$ and $j_h$ are twice differentiable, by showing the differentiability of the control-to-state mapping $G$ and its discretized  version $G_h$ in the appendix.

In Section \ref{sec-2}, we explain the DG discretization of the ODEs and the OCP. Then we present the main results for the semi-discretized case and provide some preliminary results. In Section \ref{sec-3}, the adjoint problems are studied. Section \ref{sec-4} is devoted to study the second order analysis of the cost functionals $j$ and $j_h$. In Section \ref{sec-5}, we prove the existence of the local solution and obtain the convergence rate for the semi-discretized case. Section \ref{sec-full} is devoted to establish the existence and convergence results for the fully discretized case. Finally, in Section \ref{sec-6}, we perform several numerical experiments for linear and nonlinear OCPs.  In Appendix \ref{app_d}, we obtain first and second order derivatives of the control-to-state mapping $G$.  Appendix \ref{app_c} is devoted to prove a Gronwall-type inequality for the discretization of the ODEs \eqref{main_eq2} involving the control variable. It is  used in Appendix \ref{app_e} to establish the differentiability of the discrete control-to-state mapping $G_h$ and obtain the derivatives. In Appendix \ref{app_a}, we prove Lemma \ref{lem-30} and Lemma \ref{lem-35}, which reformulate the first derivatives of the cost functionals in terms of the adjoint states. In Appendix \ref{app_b}, we derive the formulas on the second order derivatives of the cost functionals. 

%
%
%
%

\section{DG formulation}\label{sec-2}
In this section, we describe the approximation of the OCP \eqref{main_eq1}-\eqref{main_eq2} with the DG method, and then we state the main results on the semi-discrete case. First we consider the discretization of the following ODEs:
\begin{equation}\label{eq-ode}
\left\{\begin{array}{lll}
x' (t)= F(t,x(t)), \quad t \in (0,T),
\\
x(0)=x_0,
\end{array}\right.
\end{equation}
where $x: [0,T] \rightarrow \mathbb{R}^d$, $F:(0,T) \times \mathbb{R}^d \rightarrow \mathbb{R}^d$ is uniformly Lipschitz continuous with respect to {$x$}, i.e., 
\begin{equation*}
\|F(t,u) -F(t,v)\|\leq L \|u -v\|,\quad u, v \in \mathbb{R}^d,~ t \in (0,T)
\end{equation*}
with a constant $L>0$. {By the Cauchy Lipschitz theorem, we have the existence and uniqueness of classical solution $x$ of \eqref{eq-ode}.}
\

Given an integer $N \in \mathbb{N}$ we consider a partition of $I$ into $N$-intervals $\{I_n\}_{n=1}^N$ given by $I_n = (t_{n-1}, t_n)$ with nodes $0=:t_0 < t_1 < \cdots < t_{N-1} < t_N := T$. Let $h_n$ be the length of $I_n$, i.e., $h_n = t_n - t_{n-1}$, and we set $h:=\max_{1 \leq n \leq N}h_n$. For a piecewise continuous function $\varphi :[0,T] \rightarrow \mathbb{R}^d$, we also define  
\[
\varphi^+_n := \lim_{t \rightarrow 0^{+}} \varphi(t_n + t), \quad 0 \leq n \leq N-1, \qquad \varphi^-_n := \lim_{t \rightarrow 0^{+}} \varphi(t_n - t), \quad 1 \leq n \leq N.
\]
We also denote the jumps across the nodes by  $[\varphi]_n := \varphi^+_n - \varphi^-_n$ for $1 \leq n \leq N-1$. For $r \in \mathbb{N} \cup \{0\}$ we define 
\[
X^r_h := \{ \varphi_h \in X : \varphi_h |_{I_n} \in P^r(I_n), \quad 1 \leq n \leq N\},
\]
where $P^r(I_n)$ represents the set of all polynomials of $t$ up to order $r$ defined on $I_n$ with coefficients in $\R^d$. Then the DG approximate solution $x_h$ of \eqref{eq-ode} is given as
\begin{equation}\label{eq-dode}
\sum_{n=1}^N \lt(x'(t) - F(t,x(t)), \varphi(t)\rt)_{I_n} + \sum_{n=2}^N ([x]_{n-1}, \varphi^+_{n-1}) + (x^+_0, \varphi^+_0) = (x_0, \varphi^+_0)
\end{equation}
for all {$\varphi \in X^r_h$}. Here $(\cdot , \cdot)$ denotes the inner product in $\mathbb{R}^d$, and 
\[
(A(t), B(t))_{I_n} = \int_{I_n} (A(t), B(t))\,dt
\]
for integrable functions $A, B: I_n \rightarrow \mathbb{R}^d$. 

We recall the error estimate for the DG approximation of \eqref{eq-ode} from {\cite[Corollary 3.15 \& Theorem 2.6]{SS00}.} 
\begin{thm}\label{thm-21} Let $x(t)$ be the solution of \eqref{eq-ode} such that $x \in W^{k, \infty}(I; \mathbb{R}^d)$ for some $k \geq 1$. {Suppose that $h L < 1$. Then there exists a unique DG approximate solution $x_h \in X_h^r$ to \eqref{eq-dode} of order $r \in \mathbb{N} \cup \{0\}$. Furthermore, we have}
\[
\sup_{0 \leq t \leq T}|x_h (t) - x(t)| \leq C h^{\min\{r+1, k\}}\|x\|_{W^{k, \infty} (I; \mathbb{R}^d)},
\]
where $C>0$ is determined by $L$, $T$, and $r$.
\end{thm}
Now, for given $u \in \mathcal{U}$, we consider the approximate solution $x \in X^r_h$ of the control problem \eqref{main_eq2} satisfying
 \begin{equation}\label{eq-dg}
\sum_{n=1}^N \lt(x'(t) - f(t,x(t),u(t)), \varphi(t)\rt)_{I_n} + \sum_{n=2}^N ([x]_{n-1}, \varphi^+_{n-1}) + (x^+_0, \varphi^+_0) = (x_0, \varphi^+_0)
\end{equation}
for all $\varphi \in X^r_h$.

Throughout the paper, we will consider local solutions $\bar{u}$ to \eqref{main_eq3} satisfying the following non-degeneracy condition.

\noindent \textbf{Assumption 1.} Let $\bar{u} \in \U_{ad}$ be the local solution of \eqref{main_eq1}. We assume that it satisfies 
\begin{equation}\label{ass_1}
j'' (\bar{u})(v, v) \geq \gamma \|v\|_{L^2(I)}^2 \quad  \forall \,  v \in \U
\end{equation}
for some $\gamma>0$.

 In addition, we assume that $\bar{u} \in \mathcal{U}_{ad}$ has bounded total variation, i.e., $V(\bar{u}) \leq R/2$ for a fixed value $R >0$. Here the total variation $V (f)$ for $f \in L^{\infty}(0,T)$ is defined as
\[
V(f) := \sup_{P} \sum_{j=0}^{n}|f(x_j) - f(x_{j+1})|,
\]
where $P$ is any partition $P = \{ 0=x_0 < x_1 <x_2 < \cdots < x_n < x_{n+1}=T\}$. 

Considering a discrete control-to-state mapping $G_h : \U \to X^r_h$, $u \mapsto G_h(u)$, where $G_h(u)$ is the solution of \eqref{eq-dg}, we introduce the discrete cost functional $j_h : \U \to \R_+, u \mapsto J(u, G_h(u))$. We consider the following discretized version of \eqref{main_eq1}:
\bq\label{main_eq4}
\min_{u \in \U_{ad}\cap \mathcal{V}_R} j_h(u), 
\eq
where
\[
\mathcal{V}_{R} = \{ u \in \mathcal{U}~:~ V(u) \leq R\}.
\]
We now define the local solution to \eqref{main_eq4} as follows.
\begin{defn} A control $\bar u_h \in \U_{ad} \cap \mathcal{V}_R$ is called a local solution of \eqref{main_eq4}  if there exists an $\delta>0$ such that $j_h(u) \geq j_h(\bar u_h)$ holds for all $u \in \U_{ad}\cap \mathcal{V}_R$ with $\|u - \bar u_h\|_{L^2 (I)} \leq \delta$.
\end{defn}

In the first main result, we prove the existence of the local solution to the approximate problem \eqref{main_eq4}.
\begin{thm}\label{thm-1}
Let $\bar{u} \in \U_{ad}\cap \mathcal{V}_{R/2}$ be a local solution of \eqref{main_eq1} satisfying \mbox{Assumption 1}. Then, there are constants $\epsilon >0$ and $h_0 >0$ such that for $h \in (0, h_0)$ the approximate problem  \eqref{main_eq4} has a local solution $\bar{u}_h \in \mathcal{U}_{ad} \cap \mathcal{V}_R$ satisfying $\|\bar{u}_h - \bar{u}\|_{L^2 (I)} < \ep$.  
\end{thm}
The second main result is the following convergence estimate of the approximate solutions. 
\begin{thm}\label{thm-2}Let $\bar{u} \in \U_{ad}\cap \mathcal{V}_{R/2}$ be a local solution of \eqref{main_eq3} satisfying \mbox{Assumption~1}, let  $\bar{u}_h$ be the approximate solution found in Theorem \ref{thm-1}, and let $\lambda (\bar{u})$ be the adjoint state defined in Definition \ref{def-3-1} below. Assume that the state $\bar{x}=G(\bar{u})$ belongs to $W^{k_1, \infty}(I; \mathbb{R}^d)$ and the adjoint state $\lambda (\bar{u})$ belongs to $W^{k_2, \infty}(I; \mathbb{R}^d)$ for some $k_1, k_2 \geq 1$. Then we have
\begin{equation*}
\|\bar{u}- \bar{u}_h\|_{L^2(I)} = O(h^{\min\{r+1,k_1, k_2\}}).
\end{equation*}
\end{thm}
The above result establishes  the error estimate concerning the discretization of the ODEs in the OCPs. On the other hand, to implement a numerical computation to the OCP \eqref{main_eq3}, one need also consider an approximation of the control space with a finite dimensional space. In Section \ref{sec-full}, we will see that the proof of Theorem \ref{thm-2} can be extended to obtain the error analysis incorporating the discretization of the control space.

\section{Adjoint states}\label{sec-3}
This section is devoted to study the adjoint states to the OCP \eqref{main_eq1} and its discretized version \eqref{main_eq4}.

We introduce a bilinear form $b(\cdot,\cdot)$ for $x \in {W^{1,\infty} (0,T)}$ and  $\varphi \in X$ by
\begin{equation}\label{eq-3-10}
b(x,\varphi) := \int_0^T x'(t) \cdot \varphi(t)\,dt.
\end{equation}
Then, for a fixed control $u \in \U$ and initial data $x_0 \in \mathbb{R}^{d}$, a weak formulation of \eqref{main_eq2} can be written as 
\begin{equation}\label{eq-3-10}
b(x,\varphi) = \int_0^T f(t, x(t),u(t))\cdot\varphi(t)\,dt
\end{equation}
for all $\varphi \in X$ with $x(0) = x_0$. 

\begin{defn}\label{def-3-1} For a control $u \in \U$, we define the adjoint state $\lambda = \lambda(u) \in W^{1,\infty}(0,T)$ as the solution to 
\begin{equation}\label{eq-1-51}
\lambda' (t) = - \partial_x f (t,x(t), u(t)) \lambda (t) + \partial_x g (t,x(t), u(t))
\end{equation}
{with $\lambda(T) =0$}. It satisfies the weak formulation
\begin{equation}\label{eq-1-4}
b({\varphi}, \lambda) = ({\varphi}, \partial_x f (\cdot, x, u) \lambda - \partial_x g(\cdot, x, u))_{L^2(I)}
\end{equation} 
for all $\varphi \in X$ with $\lambda (T)=0$. 
\end{defn}
For $u, v \in \mathcal{U}$, the derivative of $j$ at $u$ in the direction $v$ is defined by 
\begin{equation*}
j' (u) v := \lim_{t \rightarrow 0^{+}} \frac{j(u+tv) -j (u)}{t}.
\end{equation*}
It is well-known that the derivative of the cost functional can be calculated with the adjoint state, as described below. 
\begin{lem}\label{lem-30}
We have
\begin{equation}\label{eq-1-2}
j'(u)(v) = \lt(\pa_u g(\cdot, x,  u) - \pa_u f(\cdot, x,  u)\lambda( u), v\rt)_I
\end{equation}
for all $v \in \U$, where $x = G(u)$.
\end{lem}
\begin{proof} For the completeness of the paper, we give the proof in {Appendix \ref{app_a}.}
\end{proof}
Next we describe the adjoint problem for the approximate problem \eqref{main_eq4}. For ${x}, \varphi \in X^r_h$, we define 
\bq\label{def_B}
B(x, \varphi) := \sum_{n=1}^N (x', \varphi)_{I_n} + \sum_{n=2}^N ([x]_{n-1}, \varphi^+_{n-1}) + (x^+_0, \varphi^+_0).
\eq
For approximate solution $x_h = G_h (u) \in X^r_h$, the equation \eqref{eq-dg} with control $u \in \U$ can be written as
\bq\label{dg_b}
B(x_h,\varphi) = \lt(f(\cdot,x_h,u),\varphi \rt)_I + (x_0, \varphi^+_0) \quad \forall \, \varphi \in X^r_h.
\eq
Now we define the adjoint equation for the approximate problem \eqref{main_eq4}.
\begin{defn}The adjoint state $\lambda_h = \lambda_h(u) \in X^r_h$ is defined as the solution of the following discrete adjoint equation:
\begin{equation}\label{eq-2-1}
B(\varphi, \lambda_h) = (\varphi, \pa_x f(\cdot,x_h,u)\lambda_h - \pa_x g(\cdot,x_h, u))_I  \quad \forall \, \varphi \in X^r_h.
\end{equation}
\end{defn}
In Appendix \ref{app_a}, we briefly explain how the adjoint equation \eqref{eq-2-1} can be derived from the Lagrangian related to \eqref{main_eq4}. We also have an analogous result to Lemma \ref{lem-30}.
\begin{lem}\label{lem-35}
We have
\bq\label{eqn_jk0}
j'_h(u)(v) = (\pa_u g(\cdot, x_h, u) - \pa_u f(\cdot, x_h, u)\lambda_h, v)_I  \quad \forall \,v \in \U,
\eq
where $x_h = G_h(u)$.
\end{lem}
\begin{proof}
The proof is given in {Appendix \ref{app_a}.}
\end{proof}
In order to prove the main results in Section 2, we shall use the following lemma.
\begin{lem}\label{lem_lamku}Let $u \in \mathcal{U}$. Suppose that $x= G(u) \in W^{k_1,\infty} (I;\R^d)$ and $\lambda = \lambda (u) \in W^{k_2,\infty}(I;\R^d)$ for some $k_1, k_2 \geq 1$. Then we have
\begin{equation}\label{eq-3-3}
\| \lambda ({u}) - \lambda_h ({u})\|_{L^2(I)} = O(h^{\min \{k_1, k_2, r+1\}}).
\end{equation}
\end{lem}
\begin{proof}
We recall from \eqref{eq-1-4} and \eqref{eq-2-1} that $\lambda =\lambda ({u})$ solves
\begin{equation}\label{eq-3-30}
b(\varphi, \lambda) - (\varphi, \partial_x f (\cdot, {x}, {u})\lambda)_{L^2(I)} =   -(\varphi,  \pa_x g(\cdot,  x,  u))_I,
\end{equation}
and $\lambda_h =\lambda_h ({u})$ solves
\begin{equation}\label{eq-3-31}
\begin{split}
&B(\varphi, \lambda_h) - (\varphi, \partial_x f (\cdot, {x}, {u}) \lambda_h)_{L^2(I)}
\cr
&\quad = - (\varphi, \partial_x g(\cdot, {x}_h, {u}))_{L^2(I)}  + (\varphi, ( \partial_x f (\cdot, {x}_h, {u}) - \partial_x f (\cdot, {x}, {u})) \lambda_h)_{L^2(I)}\quad \forall\,\varphi \in X_h^{r}. 
\end{split}
\end{equation}
Here $x \in G({u}) \in X$ and ${x}_h = G_h ({u}) \in X_h$. The estimate of ${x}-{x}_h$ is induced from Theorem \ref{thm-21} as follows:
\begin{equation}\label{eq-3-2}
\| {x}- {x}_h\|_{L^{\infty}(I)} = O(h^{\min\{k_1, r+1\}}) \|{x}\|_{W^{k_1, \infty}(I)}.
\end{equation}
As an auxiliary function, we consider $\zeta_h \in X_h$ solving 
\begin{equation}\label{eq-3-32}
B({\varphi}, \zeta_h ) - ({\varphi}, \partial_x f(\cdot, x,u) \zeta_h)_{I} = - ({\varphi}, \partial_x g(\cdot, x,u))_{I} \quad \forall\,\varphi \in X_h^r,
\end{equation}
which is the DG discretization of \eqref{eq-3-30} in a backward way (see Lemma \ref{lem-3-6} below). Then, by Theorem \ref{thm-21}, we have
\begin{equation}\label{eq-1}
\| \zeta_h - \lambda \|_{L^{\infty}(I)} = O (h^{\min\{k_2, r+1\}}) \|\lambda\|_{W^{k_2, \infty}(I)}.
\end{equation}
By \eqref{eq-3-2}, we obtain
\begin{equation*}
\partial_x g(\cdot, {x}, {u}) - \partial_x g(\cdot, {x}_h, {u}) = O(h^{\min\{k_1, r+1\}}) 
\end{equation*}
and
\begin{equation*}
 ( \partial_x f (\cdot, {x}_h, {u}) - \partial_x f (\cdot, {x}, {u})) \lambda_h ({u}) = O(h^{\min\{k_1, r+1\}}).
\end{equation*}
Combining these estimates with \eqref{eq-3-31} and \eqref{eq-3-32} we find
\begin{equation*}
B(\varphi, \lambda_h - \zeta_h) = (\varphi, \partial_x f (\cdot, x, u) (\lambda_h - \zeta_h))_{I} + (\varphi, R(t))_{I}\quad \forall\,\varphi \in X_h^{r},
\end{equation*}
where $R: I \rightarrow \mathbb{R}^d$ satisfies $\|R(t)\| = O(h^{\min\{k_1, r+1\}})$. 
This, together with Lemma \ref{lem-d-4}, yields
\begin{equation*}
\|\lambda_h - \zeta_h\|_{L^{\infty}(I)} = O(h^{\min\{k_1, r+1\}}).
\end{equation*}
Combining this estimate with \eqref{eq-1}, we find that
\begin{equation*}
\|\lambda_h  - \lambda \|_{L^{\infty}(I)} = O(h^{\min\{k_1,k_2, r+1\}}),
\end{equation*}
which completes the proof. 
\end{proof}
With abusing a notation for simplicity, let us define $J$ as the interval $I$ given a partition $0=s_0 < s_1 < \cdots < s_{N-1} < s_N =T$ with $s_j = t_{N-j}$. Also we  set $X_{h,J}^{r}$ as the DG space $X_h^r$ with the new partition. Then we have the following lemma.
\begin{lem}\label{lem-3-6} Assume that $\lambda \in X_h^r$ is a solution to 
\begin{equation*}
B(\phi, \lambda) = (\phi, F(t, \lambda) )_{I}\quad \forall~\phi \in X_h^r.
\end{equation*}
Then $W:I \rightarrow \mathbb{R}^d$ defined by $W(t) = \lambda (T-t)$ for $t \in I=[0,T]$ satisfies
\begin{equation*}
B(W, \psi) = (F(t,W), \psi)_{I}\quad \forall~\psi \in X_{h,J}^r.
\end{equation*}
\end{lem}
\begin{proof}
By an integraion by parts, we have
\begin{equation*}
\begin{split}
B(\phi, \lambda) & = \sum_{n=1}^{N} (\phi', \lambda)_{I_n} + \sum_{n=2}^{N} ([\phi]_{n-1},\lambda_{n-1}^{+}) + (\phi_0^{+}, \lambda_0^{+})
\\
& = -\sum_{n=1}^{N} (\phi, \lambda' )_{I_n} -\sum_{n=1}^{N-1} (\phi_{n}^{-}, [\lambda]_{n}) + (\phi_{N}^{-}, \lambda_N^{-}),
\end{split}
\end{equation*}
which leads to 
\begin{equation}\label{eq-t-2}
 -\sum_{n=1}^{N} (\phi, \lambda' )_{I_n} -\sum_{n=1}^{N-1} (\phi_{n}^{-}, [\lambda]_{n}) + (\phi_{N}^{-}, \lambda_N^{-})= (\phi, F(t,\lambda))_{I}\quad \forall~\phi \in X_h^r.
 \end{equation}
We now observe that $W(t) = \lambda (T-t)$ satisfies $W' (t) = - \lambda' (T-t)$ and $[W]_{N-n} = - [\lambda]_{n}$. We also set $\psi (t) = \phi (T-t)$. Then $\psi \in X_{h,J}^{r}$ and we have $\phi_{n}^{-} = \psi_{N-n}^{+}$. Considering $J_n :=(s_{n-1}, s_n)$, it holds that $J_n = I_{N+1-n}$ for $1 \leq n \leq N$. Using these notations, we write \eqref{eq-t-2} as
\begin{equation*}
\sum_{n=1}^{N} (\psi, W')_{J_{N+1-n}} +\sum_{n=1}^{N-1} (\psi_{N-n}^{+}, [W]_{N-n}) + (\psi_{0}^{+}, W_0^{+})= (\psi, F(t, W))_{I}\quad \forall~\psi \in X_{h,J}^{r}.
\end{equation*}
Rearranging this, we get
\begin{equation*}
\sum_{n=1}^{N} (W', \psi)_{J_{n}} +\sum_{n=1}^{N-1} ([W]_{n},\psi_{n}^{+}) + (W_0^{+}, \psi_0^{+}) = (F(t, W),\psi)_{I}\quad \forall~\psi \in X_{h,J}^{r},
\end{equation*}
which is the desired equation $B(W, \psi) = (F(t,W), \psi)_{I}$. The proof is finished.
\end{proof}
%
%
%
%

\section{Second order analysis}\label{sec-4}
In this section, we analyze the second order condition of the functions $j$ and $j_h$, which are essential in the existence and convergence estimates in the next sections. 
%
%
%
%
\subsection{Second order condition for $j$}
We defined the solution mapping $G: \mathcal{U} \rightarrow X \cap L^{\infty}(I; \mathbb{R}^d)$ in the previous section. Here we present Lipschitz estimates for the solution mapping $G$, its derivative $G'$, and the solution to the adjoint equation \eqref{eq-1-4}. 
\begin{lem}\label{lem_e1}  There there exists $C>0$ such that for all $u, \hat u \in \U_{ad}$ and $v \in \U$ we have
\[
\|G(u) - G(\hat u)\|_{L^\infty(I)} \leq C\|u - \hat u\|_{L^2(I)}, \quad \|G'(u)v - G'(\hat u)v\|_{L^{\infty}(I)} \leq C\|u - \hat u\|_{L^2(I)}\|v\|_{L^{2}(I)},
\]
and
\[
\|\lambda(u) - \lambda(\hat u)\|_{L^\infty (I)} \leq C\|u - \hat u\|_{L^2 (I)}.
\]
\end{lem}
\begin{proof}
Let us denote by $x = G(u)$ and $\hat x = G(\hat u)$. Then it follows from \eqref{eq-3-10} that
\begin{equation}\label{eq-4-60}
(x - \hat x)' (t) = f(t,x(t),u(t)) -  f(t,\hat x (t), \hat u (t)).
\end{equation}
By \eqref{eq-1-10}, there exists a constant $C>0$ such that
\begin{equation*}
\begin{split}
\left|f(t,x(t),u(t)) - f(t,\hat x(t), \hat{u}(t))\right| &\leq C|\hat{x}(t) - x(t)| + C |\hat{u}(t) - u(t)|.
\end{split}
\end{equation*}
Using this estimate and applying the Gronwall inequality in \eqref{eq-4-60}, we get the inequality
\begin{equation*}
\|x - \hat{x}\|_{L^{\infty}(I)} \leq C \|u-\hat{u}\|_{L^{1}(I)} \leq C \|u-\hat{u}\|_{L^{2}(I)} .
\end{equation*}
This gives the first inequality. For the second one, if we set $y = G'(u)v$ and $\hat y = G'(\hat u)v$, then we find from Lemma \ref{lem-d-1} that
$$\begin{aligned}
(y - \hat y)' (t) &= \pa_x f(t,x(t),u(t))( y - \hat y)(t) + (\pa_x f(t,x,u) - \pa_x f(t,\hat x,\hat u))\hat y (t)
\cr
&\quad + (\pa_u f(t,x,u) - \pa_u f(t,\hat x,\hat u))v (t).
\end{aligned}$$
This together with the first assertion above yields
$$\begin{aligned}
\|y - \hat y\|_{L^{\infty}(I)} &\leq C\|(\pa_x f(\cdot,x,u) - \pa_x f(\cdot,\hat x,\hat u))\hat y\|_{L^{1}(I)}\cr
&\quad +  C\|(\pa_u f(\cdot,x,u) - \pa_u f(\cdot,\hat x,\hat u))v\|_{L^{1}(I)}\cr
&\leq C\lt(\|x - \hat x\|_{L^2(I)} + \|u - \hat u\|_{L^2(I)}\rt)\|v\|_{L^{2}(I)}\cr
&\leq C\|u - \hat u\|_{L^2(I)}\|v\|_{L^{2}(I)}.
\end{aligned}$$
For notational simplicity, we denote by $\lambda = \lambda(u)$ and $\hat \lambda = \lambda(\hat u)$. Then, we get
$$\begin{aligned}
-(\lambda - \hat\lambda)' (t) &= \pa_x f(\cdot,x,u)(\lambda - \hat\lambda)(t) + (\pa_x f(\cdot,x,u) - \pa_x f(\cdot,\hat x , \hat u))(t)
\cr
&\quad -  (\pa_x g(\cdot,x,u) - \pa_x g(\cdot, \hat x, \hat u))(t), ~ t\in (0,T),
\end{aligned}$$
with $(\lambda - \hat{\lambda}) (T) =0$.
By applying the Gronwall inequality in a backward way, we obtain
$$\begin{aligned}
\|\lambda - \hat\lambda\|_{L^\infty(I)} &\leq C\|(\pa_x f(\cdot,x,u) - \pa_x f(\cdot,\hat x , \hat u))\hat\lambda\|_{L^1(I)} \cr
&\quad + C\|\pa_x g(\cdot,x,u) - \pa_x g(\cdot, \hat x, \hat u)\|_{L^1(I)}\cr
&\leq C(\|\hat \lambda\|_{L^\infty(I)} + 1) \lt( \|x - \hat x\|_{L^\infty(I)} + \|u - \hat u\|_{L^2(I)} \rt)\cr
&\leq C\|u - \hat u\|_{L^2(I)},
\end{aligned}$$
where we used 
\[
\|\hat \lambda\|_{L^\infty(I)} \leq C\|\pa_x g\|_{L^\infty (I)}
\]
due to \eqref{eq-1-51} and $\hat\lambda(T) = 0$. This completes the proof.
\end{proof}
We now show that the second order condition of $j$ holds near the optimal local solution $\bar{u} \in \U_{ad}$. 
\begin{lem}\label{lem-41}
Suppose that $\bar{u} \in \U_{ad}$ satisfies Assumption 1. Then there exists $\ep>0$ such that
\[
j'' (u)(v, v) \geq \frac{\gamma}{2} \|v\|_{L^2(I)}^2
\]
holds for all $v \in \U$ and all $u \in \U_{ad}$ with $\|u - \bar{u}\|_{L^2 (I)} \leq 2\ep$. Here $\gamma > 0$ is  given in \eqref{ass_1}.
\end{lem}
\begin{proof}Let $y (t) = G' (u)v$ and $y(\bar{u})(t)= G' (\bar{u}) v$. By using Lemma \ref{lem_2j}, we find
$$\begin{aligned}
& {j}'' (u)(v,v) - {j}'' (\bar{u}) (v,v)
\\
& = -\int_0^{T} \lambda(t)\left(\frac{\partial^2 f}{(\partial x)^2} (t,x,u) y^2(t) + 2 \frac{\partial^2 f}{\partial x \partial u} (t,x,u) y(t)v(t) + \frac{\partial^2 f}{(\partial u)^2} (t,x,u) v^2(t)\right)dt 
\\
&\quad + \int_0^{T} \frac{\partial^2 g}{(\partial x)^2} (t,x,u) y^2 (t) + 2 \frac{\partial^2 g}{\partial x \partial u} (t,x,u) y(t) v(t) + \frac{\partial^2 g}{(\partial u)^2} (t,x,u) v^2 (t)\,dt
\\
&\quad +\int_0^{T}\bar\lambda(t) \left( \frac{\partial^2 f}{(\partial x)^2} (t,\bar{x},u) \bar{y}^2(t) + 2 \frac{\partial^2 f}{\partial x \partial u} (t,\bar{x},u) \bar{y}(t) v(t) + \frac{\partial^2 f}{(\partial u)^2} (t,\bar{x},u) v^2(t) \right) dt
\\
&\quad - \int_0^{T} \frac{\partial^2 g}{(\partial x)^2} (t,\bar{x},u) {\bar{y}}^2 (t) + 2 \frac{\partial^2 g}{\partial x \partial u} (t,\bar{x},u) \bar{y} (t) v(t) + \frac{\partial^2 g}{(\partial u)^2} (t,\bar{x},u) v^2 (t)\,dt,
\end{aligned}$$
where we denoted by $\lambda(t) := \lambda(u)(t)$, $x(t) := G(u)(t)$, $\bar\lambda(t) := \lambda(\bar u)(t)$, and $\bar x(t) := G(\bar u)(t)$. On the other hand, it follows from Lemma \ref{lem_e1} that 
\begin{align}\label{est_er}
\begin{aligned}
\|x-\bar{x}\|_{L^\infty(I)} &\leq C\|u -\bar{u}\|_{L^2(I)}, \quad \|y-\bar{y}\|_{L^{\infty}(I)}\leq C\|u -\bar{u}\|_{L^2(I)}\|v\|_{L^2(I)}, \cr 
\|y\|_{L^{\infty}(I)} &\leq C\|v\|_{L^2(I)}, \quad \|\lambda\|_{L^\infty(I)} + \|\bar\lambda\|_{L^\infty(I)} \leq C\|\pa_x g\|_{L^\infty (I)}, \quad \mbox{and} \cr
\|\lambda-\bar\lambda\|_{L^\infty(I)} &\leq C\|u -\bar{u}\|_{L^2(I)}.
\end{aligned}
\end{align}
This together with the following estimate
$$\begin{aligned}
\int_0^T |y^2(t) - \bar y^2(t) |\,dt &\leq \int_0^T |y(t) + \bar y(t) ||y(t) - \bar y(t) |\,dt \cr
&\leq \|y - \bar y\|_{L^2(I)}\lt(\|y\|_{L^2(I)} + \|\bar y\|_{L^2(I)} \rt)\cr
&\leq C\|u -\bar{u}\|_{L^2(I)}\|v\|_{L^2(I)}^2
\end{aligned}$$
yields
\begin{equation*}
|(j'' (u) (v,v) - j'' (\bar{u}) (v,v))|\leq C\|u-\bar{u}\|_{L^2 (I)} \|v\|_{L^2(I)}^2.
\end{equation*}
Combining this with \eqref{ass_1} we have
$$\begin{aligned}
j'' (u) (v,v) &= j'' (\bar{u}) (v,v) + (j'' (u) (v,v) - j'' (\bar{u}) (v,v))
\\
&\geq \gamma \|v\|_{L^2(I)}^2  - C\|u-\bar{u}\|_{L^2 (I)} \|v\|_{L^2(I)}^2.
\end{aligned}$$
By choosing $\ep= \frac{\gamma}{4C}>0$ here, we obtain the desired result.

\end{proof}
As a consequence of this lemma, we have the following result. 
\begin{thm}\label{thm-31}
Let $\bar{u} \in \U_{ad}$ satisfy the first optimality condition and Assumption 1. Then, there exist a constant $\ep >0$ such that 
\[
j(u) \geq j (\bar{u}) + \frac{\gamma}{2} \|u-\bar{u}\|_{L^2(I)}^2
\]
for any $u \in \U_{ad}$ with $\|u-\bar{u}\|_{L^2 (I)} \leq 2\ep$.
\end{thm}
\begin{proof}
Choose $\ep >0$ as in Lemma \ref{lem-41}. By Taylor's theorem, we get
\[
j(u) = j(\bar{u}) + j' (\bar{u}) (u- \bar{u}) + \frac{1}{2}j'' (\bar{u}_{s}) (u-\bar{u}, u -\bar{u}),
\]
where $\bar u_{s} = \bar{u} + s(u-\bar{u})$ for some $s \in [0,1]$. On the other hand, by the first {optimality} condition, we have
\begin{equation}\label{eq-11}
j'(\bar u)(u - \bar u) \geq 0  \quad \forall \,u \in \U_{ad}.
\end{equation}
Moreover, we also find
\[
\|\bar u-\bar u_{s}\|_{L^2 (I)} \leq s\|u-\bar{u}\|_{L^2 (I)} \leq 2\ep.
\]
Using these observations and Lemma \ref{lem-41}, we conclude
\[
j(u) \geq j(\bar{u}) + \frac{\gamma}{2} \|u-\bar{u}\|_{L^2(I)}^2.
\]
\end{proof}

%
%
%
%
\subsection{Second order condition for $j_h$}
In this part, we investigate the second order condition for the discrete cost functional $j_h$. {Similarly as in the previous subsection}, we first provide the Lipschitz estimates for $G_h$ and the discrete adjoint state. 
\begin{lem} Let $u, \hat u \in \U_{ad}$ and $v \in \U$ be given. Then, there exists $C>0$, independent of $h \in (0,1)$, such that
$$\begin{aligned}
\|G_h(u) - G_h(\hat u)\|_{L^\infty(I)} &\leq C\|u - \hat u\|_{L^2(I)}, \cr
\|G'_h(u)v - G'_h(\hat u)v\|_{L^2(I)} &\leq C\|u - \hat u\|_{L^2(I)}\|v\|_{L^2(I)},
\end{aligned}$$
and
\[
\|\lambda_h(u) - \lambda_h(\hat u)\|_{L^\infty(I)} \leq C\|u - \hat u\|_{L^2(I)}.
\]
\end{lem}
\begin{proof}
{The first and the third assertions are proved in Lemma \ref{lem-3-5}.} The second estimate is proved in Lemma \ref{lem-d-32}. 
\end{proof}
\begin{lem}\label{lem_yku} For $u \in \U_{ad}$, let $x = G(u)$ be given by the solution of the state equation \eqref{main_eq2}, and let $y = G'(u)v$ for $v \in \U$. Let $x_h = G_h(u)$ be the solution of the discrete state equation \eqref{dg_b}, and let $y_h = G'_h(u)v$. Then we have
\[
\|y_h - y\|_{L^{\infty}(I)} \leq Ch \|v\|_{L^2(I)}.
\]
\end{lem}
\begin{proof} Define $\tilde y :[0,T] \rightarrow \mathbb{R}^d$ by the solution to 
\bq\label{eq_b}
\tilde{y}' (t) = \pa_x f(\cdot, x_h,u)\tilde y (t) + \pa_u f(\cdot,x_h,u)v (t),\quad \tilde{y}(0) =0.
\eq
Recall from Lemma \ref{lem-d-1} that $y$ satisfies
\[
y' (t) = \pa_x f(\cdot, x,u)y + \pa_u f(\cdot,x,u)v, \quad y(0)=0.
\]
Combining these two equations, we get
$$\begin{aligned}
(\tilde y - y)' (t)&= \pa_x f(t,x_h,u)(\tilde y - y)(t) +\lt(\pa_x f(t,x_h,u) - \pa_x f(t, x,u) \rt)y(t)
\\&\quad + (\pa_u f(t,x_h,u) - \pa_u f(t, x,u)  v (t).
\end{aligned}$$
Using the Gronwall inequality here with \eqref{est_er} and \eqref{eq-3-2}, we find that
\begin{equation}\label{eq-4-61}
\begin{split}
\|\tilde y - y\|_{L^{\infty}(I)}&\leq C\|x_h - x\|_{L^\infty(I)}\lt( \|y\|_{L^2(I)}  + \|v\|_{L^2(I)}\rt)\cr
&\leq C\|x_h - x\|_{L^\infty(I)}\|v\|_{L^2(I)}\cr
&\leq Ch\|v\|_{L^2(I)}.
\end{split}
\end{equation}
On the other hand, $y_h$ satisfies
\[
B(y_h,\varphi) = (\pa_x f(\cdot, x_h,u)y_h + \pa_u f(\cdot,x_h,u)v, \varphi)_I \quad \forall\, \varphi \in X_h^r,
\]
which is the DG discretization of \eqref{eq_b} in a backward way in view of Lemma \ref{lem-3-6}. Thus, we may use Theorem \ref{thm-21} to obtain the following error estimate:
\[
\|\tilde y - y_h\|_{L^{\infty}(I)} \leq Ch\|v\|_{L^2(I)}.
\]
This, together with \eqref{eq-4-61} gives us the estimate
\[
\|y_h - y\|_{L^{\infty}(I)} \leq \|\tilde y - y\|_{L^{\infty}(I)}+ \|\tilde y - y_h\|_{L^{\infty}(I)} \leq Ch\|v\|_{L^2(I)}.
\]
The proof is finished.
\end{proof}
\begin{lem}\label{lem-32} For  $\ep >0$ given in Lemma \ref{lem-41}, there exists $h_0 >0$ such that for $h \in (0, h_0 )$ we have the following inequality
\[
j''_h (u) (v,v) \geq \frac{\gamma}{4}\|v\|_{L^2(I)}^2, \quad v \in \U
\] 
for any $u \in \U_{ad}$ satisfying $\|u -\bar{u}\|_{L^2(I)} \leq \ep$. 
\end{lem}
\begin{proof} We first claim that
\bq\label{est_erj}
|j'' (u)(v,v) - j''_h (u) (v,v)| \leq Ch\|v\|_{L^2(I)}^2
\eq
for $h > 0$ small enough, where $C > 0$ is independent of $h$. Let $x(t) = G(u)(t)$, $\lambda (t) = \lambda(u) (t)$, $x_h (t)= G_h (u)(t)$, and $\lambda_h (t) = \lambda_h (u)(t)$. Also we let $y = G' (u)v$ and $y_h = {G_h}' (u)v$.  It follows from Lemmas \ref{lem_2j} and \ref{lem_2jk} that 
$$\begin{aligned}
& j'' (u)(v,v) - j''_h (u) (v,v)
\\
& = -\int_0^{T} \lambda(t)\left(\frac{\partial^2 f}{(\partial x)^2} (t,x,u) y^2(t) + 2 \frac{\partial^2 f}{\partial x \partial u} (t,x,u) y(t)v(t) + \frac{\partial^2 f}{(\partial u)^2} (t,x,u) v^2(t)\right)dt 
\\
&\quad + \int_0^{T} \frac{\partial^2 g}{(\partial x)^2} (t,x,u) y^2 (t) + 2 \frac{\partial^2 g}{\partial x \partial u} (t,x,u) y(t) v(t) + \frac{\partial^2 g}{(\partial u)^2} (t,x,u) v^2 (t)\,dt
\\
&\quad +\int_0^{T}\lambda_h(t) \left( \frac{\partial^2 f}{(\partial x)^2} (t,x_h,u) y_h^2(t) + 2 \frac{\partial^2 f}{\partial x \partial u} (t,x_h,u) y_h(t) v(t) + \frac{\partial^2 f}{(\partial u)^2} (t,x_h,u) v^2(t) \right) \, dt
\\
&\quad - \int_0^{T} \frac{\partial^2 g}{(\partial x)^2} (t,x_h,u) y_h^2 (t) + 2 \frac{\partial^2 g}{\partial x \partial u} (t,x_h,u) y_h (t) v(t) + \frac{\partial^2 g}{(\partial u)^2} (t,x_h,u) v^2 (t)\,dt.
\end{aligned}$$
In order to show \eqref{est_erj}, by using a similar argument as in the proof of Lemma \ref{lem-41}, it suffices to show that there exists $C>0$, independent of $h$, such that
\bq\label{est_er2}
\|x-x_h\|_{L^\infty(I)} \leq Ch, \quad \|y-y_h\|_{L^{\infty}(I)} \leq Ch\|v\|_{L^2(I)}, \quad \|y_h\|_{L^2(I)} \leq C\|v\|_{L^2(I)},
\eq
\bq\label{est_er3}
\|\lambda_h\|_{L^\infty(I)} \leq C, \quad \|\lambda-\lambda_h\|_{L^\infty(I)} \leq Ch,
\eq
and
\[
\int_0^T |y^2(t) - y_h^2(t)|\,dt \leq Ch\|v\|_{L^2(I)}^2.
\]
The first and second inequalites in \eqref{est_er2} hold due to Theorem \ref{thm-21} and Lemma \ref{lem_yku}. For the third one in \eqref{est_er2} is proved in \eqref{eq-d-61}. By Lemma \ref{lem_lamku}, the second inequality in \eqref{est_er3} holds. We also find
\[
\|\lambda_h\|_{L^\infty(I)} \leq \|\lambda-\lambda_h\|_{L^\infty(I)} + \|\lambda\|_{L^\infty(I)} \leq Ch + C \leq C,
\]
which asserts the first inequality in \eqref{est_er3}. Finally, we obtain
$$\begin{aligned}
\int_0^T |y^2(t) - y_h^2(t)|\,dt &\leq \int_0^T |y(t) + y_h(t)||y(t) - y_h(t)|\,dt \cr
&\leq \|y(t) - y_h(t)\|_{L^2(I)}\lt(\|y\|_{L^2(I)} + \|y_h\|_{L^2(I)} \rt)\cr
&\leq Ch\|v\|_{L^2(I)}^2,
\end{aligned}$$
due to \eqref{est_er2}. All of the above estimates enable us to prove the claim \eqref{est_erj}. This together with Lemma \ref{lem-41} yields
$$\begin{aligned}
j''_h(u)(v,v) &\geq  j''(u)(v,v) - |j''_h(u)(v,v) - j''(u)(v,v)| \cr
&\geq \frac\gamma2\|v\|_{L^2(I)}^2 - Ch\|v\|_{L^2(I)}^2 \cr
& \geq \frac\gamma4\|v\|_{L^2(I)}^2
\end{aligned}$$
for $0< h < h_0:= \gamma/(4C)$. The proof is finished.
\end{proof}
%
%
%
%

\section{Existence and Convergence results for the semi-discrete case}\label{sec-5}

We first prove the existence of the local solution to the approximate problem \eqref{main_eq4}.  

\begin{proof}[Proof of Theorem \ref{thm-1}] Choose $\ep>0$ as in Theorem \ref{thm-31}.  We consider the following set
\[
\overline{B_\ep (\bar{u})} = \{ u \in \U_{ad} :~\|u -\bar{u}\|_{L^2(I)} \leq 2\ep\},
\]
and recall the space $\mathcal{V}_{R} = \{ u \in \mathcal{U}~:~ V(u) \leq R\}.$ We will find a minimizer of $j_h$ in the space $W_{\ep, R}:=\overline{B_{\ep}(\bar{u})} \cap \mathcal{V}_R$, and then show that $\|\bar{v} - \bar{u} \|_{L^2 (I)} < \ep$. It will imply that $\bar{v}$ is a local solution to \eqref{main_eq4}.

Since $j_h$ is lower bounded on $W_{\ep, R}$, there exists a sequence $\{v_k\}_{k \in \mathbb{N}} \subset \overline{B_{\ep}(\bar{u})} \cap \mathcal{V}_R$ such that
\begin{equation}\label{eq-5-50}
\lim_{k \rightarrow \infty} j_h (v_k) =  \inf_{v \in W_{\ep,R}}  j_h (v).
\end{equation}
Moreover, since $W_{\ep,R}$ is compactly embedded in $L^p(I)$ for any $p \in [1,\infty)$, up to  a subsequence, $\{v_k\}$ converges to a function $\bar{v} \in W_{\ep, R}$ in $L^{2}(I)$ and converges a.e. to $\bar{v}$. By definition, the function $z_k := G_h (v_k) \in X_h^r$ satisfies
 \begin{equation}\label{eq-dg-r}
\sum_{n=1}^N \lt({z_k}'(t) - f(t,{z_k}(t),v_k(t)), \varphi(t)\rt)_{I_n} + \sum_{n=2}^N ([{z_k}]_{n-1}, \varphi^+_{n-1}) + ({z_k}^+_0, \varphi^+_0) = ({z_k}_0, \varphi^+_0)
\end{equation}
for all $\varphi \in X^r_h$.
Note that $\{z_k\}_{k \in \mathbb{N}}$ is a bounded set in the finite dimensional space $X_h^r$ by Theorem \ref{thm-2} (see also Lemma \ref{lem-d-4}). Therefore we can find a subsequence such that $z_k$ converges uniformly to a function $\bar{z} \in X_h^r$. We claim that $\bar{z} = G_h (\bar{v})$. Indeed, since $v_k (t)$ converges a.e. to $\bar{v}(t)$ for $t \in I$ and $f$ is Lipsichtiz continuous, we may take a limit $k$ to infinity in \eqref{eq-dg-r} to deduce
\[
\sum_{n=1}^N \lt({\bar{z}}'(t) - f(t,{\bar{z}}(t),\bar{v}(t)), \varphi(t)\rt)_{I_n} + \sum_{n=2}^N ([\bar{z}]_{n-1}, \varphi^+_{n-1}) + (\bar{z}^+_0, \varphi^+_0) = (\bar{z}_0, \varphi^+_0)
\]
for all $\varphi \in X^r_h$. This yields that $\bar{z} = G_h (\bar{v})$, which enables us to derive
\[
\begin{split}
\lim_{k \rightarrow \infty} j_h (v_k) & = \lim_{k \rightarrow \infty} \int_0^{T} g(t, z_k (t), v_k (t)) dt
\\
&=\int_0^{T} \lim_{k \rightarrow \infty} g(t, z_k (t), v_k (t)) \,dt
\\
& = \int_0^{T} g(t, \bar{z}(t), \bar{v}(t))\,dt \cr
& = \int_0^{T} g(t, G_h (\bar{v})(t), \bar{v}(t)) \,dt \cr
&= j_h (\bar{v}).
\end{split}
\]
This together with \eqref{eq-5-50} implies that $\bar{v} \in W_{\ep,R}$ satisfies 
\[
j_h (\bar{v}) = \inf_{v \in W_{\ep, R}} j_h (v).
\]
It remains to show that the minimizer is achieved in the interior of $B_\ep (\bar{u})$. To show this, we recall that
\begin{equation*}
j(u) = J(u, G(u)) = \int_0^{T} g(t, G(u)(t), u(t)) \,dt,
\end{equation*}
and
\begin{equation*}
j_h (u) = J(u, G_h (u)) = \int_0^{T} g(t, G_h (u)(t), u(t))\, dt.
\end{equation*}
Since $\|G(u) \|_{W^{1,\infty}(I)} \leq C$ for all $u \in \mathcal{U}_{ad}$, we see from Theorem \ref{thm-21} that
\begin{equation*}
\|G_h (u) - G(u) \|_{L^{\infty}(I)} \leq Ch \|G (u) \|_{W^{1,\infty}(I)} \leq Ch. 
\end{equation*}
Combining this with the Lipshitz continuity of $G$ yields that
\begin{equation*}
|j (u) - j_h (u) | \leq Ch \quad \forall\, u \in \mathcal{U}_{ad}.
\end{equation*}
Taking $h_0 = \gamma \ep^2 /(8C)$. Using this and the estimate 
\[
j (u) \geq j(\bar{u}) + \frac{\gamma}{2} \ep^2,\quad \forall\,u \in \mathcal{U}_{ad} \quad \textrm{with} ~\ep \leq \|u -\bar{u}\|_{L^2(I)} \leq 2\ep
\]
from Theorem \ref{thm-31}, it follows that for $h \in (0, h_0)$ we have
\begin{equation}\label{eq-5-54}
j_h (u) \geq j_h (\bar{u}) + \frac{\gamma}{4} \ep^2\quad \forall\,u \in \mathcal{U}_{ad} \quad\textrm{with}~ \ep \leq \|u-\bar{u}\|_{L^2(I)} \leq 2\ep.
\end{equation}
Thus, the minimizer $\bar{v}$ is achieved in $B_\ep (\bar{u})$. It gives that $j_h (u) \geq j_h (\bar{v})$ for all $u \in \mathcal{V}_R$ with $\|u-\bar{v}\|_{L^2} \leq \ep$.
\end{proof}

We now provide the details of the convergence estimate of the approximate solutions.

\begin{proof}[Proof of Theorem \ref{thm-2}]
Analogous to \eqref{eq-11}, the discrete first order necessary optimality condition for $\bar u_h \in \U_{ad}$ reads
\[
j'_h(\bar u_h)( u - \bar u_h) \geq 0  \quad \forall \,u \in B_{\ep} (\bar{u}_h) \cap \mathcal{V}_R.
\]
Inserting here $u=\bar{u}$ and summing it with \eqref{eq-11}, we get
\begin{equation}\label{eq-3-1}
\begin{split}
0 &\leq (j' (\bar{u}) - j_h' (\bar{u}_h)) (\bar{u}_h - \bar{u}) 
\\
& = (j' (\bar{u}) - j_h' (\bar{u})) (\bar{u}_h - \bar{u}) +  (j_h' (\bar{u}) - j_h' (\bar{u}_h)) (\bar{u}_h - \bar{u}).
\end{split}
\end{equation}
Now, by applying the mean value theorem with a value $t \in (0,1)$, we have
\begin{equation}\label{eq-4-1}
\begin{split}
C \|\bar{u}_h - \bar{u}\|_{L^2(I)}^2 & \leq j_h'' (\bar{u} - t (\bar{u}- \bar{u}_h))(\bar{u}_h - \bar{u}, \bar{u}_h - \bar{u}) 
\\
&= (j_h' (\bar{u}_h) - {j_h}' (\bar{u})) (\bar{u}_h - \bar{u}) 
\\
& \leq ( {j}' (\bar{u}) -{j_h}' (\bar{u})) (\bar{u}_h - \bar{u}), 
\end{split}
\end{equation}
where we used Lemma \ref{lem-32} in the first inequality and \eqref{eq-3-1} in the second inequality. For our aim, it only remains to estimate the right hand side. 
 Let us express it using the adjoint states.
From \eqref{eq-1-2}, we have
\begin{equation}\label{eq-5-52}
j'(\bar u)(\bar u_h - \bar u) = \lt(\pa_u g(\cdot,\bar x, \bar u) - \pa_u f(\cdot,\bar x, \bar u)\lambda(\bar u), \bar u_h-\bar u\rt)_I,
\end{equation}
and it follows from \eqref{eqn_jk0} that 
\begin{equation}\label{eq-5-53}
j'_h(\bar u)( \bar{u}_h - \bar u) = (\pa_u g(\cdot,\bar x_h, \bar u) - \pa_u f(\cdot,\bar x_h, \bar u)\lambda_h (\bar u), \bar{u}_h  - \bar u)_I.
\end{equation}
Here we remind that $\bar{x}_h \in X_h^r$ denotes the solution to \eqref{eq-dg} with control $\bar{u}$ and initial data $x_0$. Combining \eqref{eq-5-52} and \eqref{eq-5-53} we find
$$\begin{aligned}
( {j}' (\bar{u}) - j_h' (\bar{u}))(\bar{u}_h -\bar{u})&=\Big(  \partial_u g(\cdot, \bar{x}, \bar{u}) - \partial_u g(\cdot, \bar{x}_h, \bar{u}), ~\bar u_h-\bar u\Big)_I
\\
&\quad - \Big(  \partial_u f(\cdot, \bar{x}, \bar{u}) \lambda (\bar{u}) - \partial_u f (\cdot, \bar{x}_h, \bar{u}) \lambda_h (\bar{u}), ~\bar u_h-\bar u\Big)_I.
\end{aligned}$$
Applying H\"older's inequality here and using \eqref{eq-1-10}, we deduce
\begin{equation}\label{eq-5-70}
\begin{split}
&( j' (\bar{u}) -j_h' (\bar{u}))(\bar{u}_h -\bar{u})\cr
&\quad \leq \|\partial_u \partial_x g\|_{L^{\infty}} \|\bar{x}-\bar{x}_h \|_{L^2(I)} \|\bar{u}_h - \bar{u}\|_{L^2(I)}\\
&\qquad + \|\lambda (\bar{u})\|_{L^{\infty}(I)}\|\partial_u f (\cdot, \bar{x}, \bar{u}) - \partial_u f (\cdot, \bar{x}_h, \bar{u}) \|_{L^2(I)} \|\bar{u}_h - \bar{u}\|_{L^2(I)}\\
&\qquad + \|\partial_u f(\cdot, \bar{x}_h, \bar{u})\|_{L^{\infty}} \|\lambda (\bar{u}) - {\lambda_{h}} (\bar{u})\|_{L^2(I)} \|\bar{u}_h - \bar{u}\|_{L^2(I)}\cr
&\quad \leq C\lt(\|\bar{x}-\bar{x}_h \|_{L^2(I)} +\|\lambda (\bar{u}) -{\lambda_{h}}(\bar{u})\|_{L^2(I)}\rt)\|\bar{u}_h - \bar{u}\|_{L^2(I)}.
\end{split}
\end{equation}
Now we apply \eqref{eq-3-3} and \eqref{eq-3-2} to get
\begin{equation}\label{eq-5-10}
( {j}' (\bar{u}) -{j_h}' (\bar{u}))(\bar{u}_h -\bar{u})\leq C h^{\min\{k_1, k_2, r+1\}} \|\bar{u}_h -\bar{u}\|_{L^2(I)}.
\end{equation}
Combining this with \eqref{eq-4-1}, we finally obtain
\[
\|\bar{u}_h - \bar{u}\|_{L^2(I)} \leq C h^{\min\{k_1, k_2, r+1\}}.
\]
This completes the proof.
\end{proof}

%
%
%
%

\section{Existence and Convergence results for the fully discrete case}\label{sec-full}
This section is devoted to the existence and convergence results for the fully discrete case. We consider a finite dimensional space $\U_h$ which discretizes the control space $\U_{ad}$, for example, the space of step functions
\[
\U_h = \{ u \in \U_{ad} \mid ~u: \textrm{piecewise constant on}~ I_k = [t_{k-1},t_k]\},
\]
or the high-order DG space $\mathcal{U}_h = X_h^r$ with $r \in \mathbb{N}$. 
\

We say that $\bar{u}_h \in \mathcal{U}_h$ is a local solution to 
\begin{equation}\label{eq-r-31}
\min_{u \in U_h} j_h (u)
\end{equation}
if there is a value $\ep >0$ such that $j_h (u) \geq j_h (\bar{u}_h)$ for all $u \in \mathcal{U}_h$ with $\|u- \bar{u}_h \|_{L^2} \leq \ep$.
\

\

We provide the existence result of local solution in the following theorem.
\begin{thm}\label{thm-6-1} Choose $\ep>0$ as in Theorem \ref{thm-31}. Let $\bar{u} \in \mathcal{U}_{ad}$ be a local solution of \eqref{main_eq3} satisfying Assumption 1. Fix any $\ep >0$. Then there exists $h_0 >0$ such that for $h \in (0, h_0 )$ problem \eqref{eq-r-31} has a local solution $\bar{u}_h \in \mathcal{U}_h$ such that $\|\bar{u}- \bar{u}_h\|_{L^2} \leq \ep$.
\end{thm} 

\begin{proof} 
By compactness and continuity, $j_h$ has a minimizer in 
\[
\overline{B_{2\ep} (\bar{u})} = \{ u \in \U_{h} :~\|u -\bar{u}\|_{L^2(I)} \leq 2\ep\},
\]
since $\mathcal{U}_h$ is finite dimensional.
Next we aim to show that the minimizer $\bar{u}_h$ satisfies 
\[
\|\bar{u}_h - \bar{u}\|_{L^2 (I)} \leq \ep.
\]
 To show this, we recall from \eqref{eq-5-54} that there is a value $h_0 >0$ such that for $h \in (0, h_0)$ we have
\begin{equation*}
j_h (u) \geq j_h (\bar{u}) + \frac{\gamma}{4} \ep^2, \quad \forall\,u \in \mathcal{U}_{ad} \quad\textrm{with}\quad \ep \leq \|u-\bar{u}\|_{L^2(I)} \leq 2\ep.
\end{equation*}
Combining this with the minimality of $\bar{u}_h$ for $j_h$ in $\overline{B_{2\ep}(\bar{u})}$, we find that $\|\bar{u}_h - \bar{u}\|_{L^2 (I)} \leq \ep$.  It then yields that
\begin{equation*}
j_h (u) \geq j_h (\bar{u}_h), \quad \forall\,u \in \mathcal{U}_h \quad \textrm{with}\quad \|u-\bar{u}_h\|_{L^2} \leq \ep.
\end{equation*}
Thus $\bar{u}_h$ is a local solution of \eqref{eq-r-31}.
\end{proof}

We establish the convergence result in the following theorem.
\begin{thm}Assume the same statements for $\bar{u}\in \mathcal{U}_{ad}$ and $\lambda (\bar{u})$ in Theorem \ref{thm-2}. In addition, suppose that there exists a projection operator $P_h : \mathcal{U} \rightarrow \mathcal{U}_h$ and a value $a >0$ such that 
\[
\|P_h \bar{u} - \bar{u}\|_{L^2 (I)} = O(h^a) \quad \mbox{for} \quad h \in (0,1).
\]
 Let $\bar{u}_h \in \mathcal{U}_h$ be a local solution to \eqref{eq-r-31} constructed in Theorem  \ref{thm-6-1}. Then the following estimate holds:
\[
\|\bar{u}_h - \bar{u}\|_{L^2 (I)} =O (h^{\min \{ r+1, k_1, k_2, a/2}\}).
\]
If we further assume that $j' (\bar{u}) =0$, then the above estimate can be improved to
\[
\|\bar{u}_h - \bar{u}\|_{L^2 (I)} = O (h^{\min \{ r+1, k_1, k_2, a}\}).
\]
\end{thm}
\begin{proof}

In this case, by the first optimality conditions on $\bar{u}$ and $\bar{u}_{h}$, we have
\[
j' (\bar{u}) (\bar{u}_h - \bar{u}) \geq 0 \quad \mbox{and} \quad {j_h}' (\bar{u}_h) (P_h \bar{u} - \bar{u}_h) \geq 0.
\]
The latter condition can be written as
$$\begin{aligned}
0 &\leq j_h' (\bar{u}_h) (\bar{u} - \bar{u}_h) + j_h' (\bar{u}_h) (P_h \bar{u} - \bar{u})= j_h' (\bar{u}_h) (\bar{u} - \bar{u}_h) +R_h,
\end{aligned}$$
where $R_h := {j_h}' (\bar{u}_h) (P_h \bar{u} - \bar{u})$.
Summing up the above two inequalities, we get
$$\begin{aligned}
0 &\leq (j' (\bar{u}) - {j_h}' (\bar{u}_h)) (\bar{u}_h - \bar{u})+R_h\\
& = (j' (\bar{u}) - j_h' (\bar{u})) (\bar{u}_h - \bar{u}) +  (j_h' (\bar{u}) - j_h' (\bar{u}_h)) (\bar{u}_h - \bar{u})+R_h,
\end{aligned}$$
i.e.,
\bq\label{est_jkk}
(j_h' (\bar{u}_h) - j_h' (\bar{u})) (\bar{u}_h - \bar{u}) \leq ( j' (\bar{u}) -j_h' (\bar{u})) (\bar{u}_h - \bar{u})+R_h.
\eq
By the assumption of the theorem, we have
\begin{equation}\label{eq-5-11}
\|R_h \|_{L^2 (I)} = O (h^a).
\end{equation}
On the other hand, by applying the mean value theorem and Lemma \ref{lem-32}, we obtain
\[
(j_h' (\bar{u}_h) - j_h' (\bar{u})) (\bar{u}_h - \bar{u}) = j_h'' (\bar{u} + t (\bar{u}- \bar{u}_h))(\bar{u}_h - \bar{u}, \bar{u}_h - \bar{u}) \geq C\|\bar{u}_h - \bar{u}\|_{L^2(I)}^2.
\]
Combining this with \eqref{est_jkk} yields
\[
\|\bar{u}_h - \bar{u}\|_{L^2(I)}^2 \leq C( j' (\bar{u}) -j_h' (\bar{u})) (\bar{u}_h - \bar{u}) + CR_h.
\]
Applying here the estimate \eqref{eq-5-10} in the previous proof, we have
\begin{equation}\label{eq-5-12}
\|\bar{u}_h - \bar{u}\|_{L^2(I)}^2 \leq C h^{\min\{k_1, k_2, r+1\}} \|\bar{u}_h -\bar{u}\|_{L^2(I)} + CR_h,
\end{equation}
which together with \eqref{eq-5-11} gives the desired estimate
\[
\|\bar{u}_h - \bar{u}\|_{L^2 (I)} =O (h^{\min \{ r+1, k_1, k_2, a/2}\}).
\]
When we further assume  $j' (\bar{u}) =0$, we have
\[
{j_h}' (\bar{u}_h) = ({j_h}' (\bar{u}_h) - {j_h}' (\bar{u})) +({j_h}' (\bar{u}) -j' (\bar{u})).
\]
Using this and the estimates in \eqref{eq-5-70}, we find
\[
\begin{split}
|R_h| = | {j_h}' (\bar{u}_h)  (P_h \bar{u} - \bar{u})| &\leq C\lt( \|\bar{u}_h -\bar{u}\|_{L^2 (I)} +  h^{\min\{k_1, k_2, r+1\}}\rt) \| P_h \bar{u} - \bar{u}\|_{L^2(I)}
\\
&\leq  C h^a \lt( \|\bar{u}_h -\bar{u}\|_{L^2 (I)} +  h^{\min\{k_1, k_2, r+1\}}\rt).
\end{split} 
\]
Inserting this into \eqref{eq-5-12} we obtain
\[
\begin{split}
\|\bar{u}_h - \bar{u}\|_{L^2(I)}^2 &\leq C h^{\min\{k_1, k_2, r+1\}} \|\bar{u}_h -\bar{u}\|_{L^2(I)}
\\
&\quad +  Ch^a ( \|\bar{u}_h -\bar{u}\|_{L^2 (I)} +  h^{\min\{k_1, k_2, r+1\}}).
\end{split}
\]
It gives the desired estimate
\[
\|\bar{u}_h - \bar{u}\|_{L^2 (I)} =O (h^{\min \{ r+1, k_1, k_2, a}\}).
\]
This concludes the desired result.
\end{proof}
%
%
%
%
\section{Numerical experiments}\label{sec-6}

In this section, we present several numerical experiments which validate our theoretical results. We employed the forward-backward DG methods \cite{DHT81} to solve the examples of the OCPs.

\subsection{Linear problem}

Let us consider the following simple one dimensional OCP, which has been used as an example \cite{VV88}, that consists of maximizing the functional
\[
J = \frac12\int_0^1 x^2(t) + u^2(t)\,dt
\]
subject to the state equation
\bq\label{lin_ex1}
x'(t) = -x(t) + u(t), \quad x(0) = 1,
\eq
and $\mathcal{U}= L^2 ([0,1])$. Using a similar idea as in Section \ref{sec-3} based on the maximum principle, we can derive the adjoint equation to the above optimal control problem:
\[
\lambda'(t) = \lambda(t) - x(t), \quad \lambda(1) = 0.
\]
Furthermore, we also find that the optimal solutions $\bar u = - \lambda$ and $\bar x$ satisfies \eqref{lin_ex1}. Thus we have the solution
\[
\bar x(t) = \frac{\sqrt2 \cosh(\sqrt2(t-1))-\sinh(\sqrt2(t-1))}{\sqrt2 \cosh(\sqrt2) + \sinh(\sqrt2)}
\]
and
\[
\bar u(t) = \frac{\sinh(\sqrt2(t-1)}{\sqrt2 \cosh(\sqrt2) + \sinh(\sqrt2)}.
\]
For fixed $r \in \mathbb{N}$, we use $X_h^r$ for the approximate space of $\mathcal{U}$. 
In Table \ref{tb1}, we report the discrete $L^2$ error between optimal solutions and its approximations for the above optimal control problem. Here $r+1$ is the number of grid points on each time interval $I_n$, and we used the equidistant points for our numerical computations. The numerical result confirms that the error is of order $h^{r+1}$ as proved in Theorem \ref{thm-2}.
\begin{table}[htp]
\caption{Discrete $L^2$ error: $\|\bar x - \bar x_h\|_{L^2(I)}$ and $\|\bar u - \bar u_h\|_{L^2(I)}$}
\begin{center}
\begin{tabular}{c|ccccc}
 \mbox{\quad} $$  \mbox{\quad}& \mbox{\quad} $h$  \mbox{\quad}&  \mbox{\quad} $\|\bar x - \bar x_h\|_{L^2(I)}$  \mbox{\quad}&  \mbox{\quad} $\|\bar u - \bar u_h\|_{L^2(I)}$ \mbox{\quad}&$\log_2 \frac{\|\bar{x}-\bar{x}_{2h}\|}{\|\bar{x}-\bar{x}_{h}\|}$\mbox{\quad}&$\log_2 \frac{\|\bar{u}-\bar{u}_{2h}\|}{\|\bar{u}-\bar{u}_{h}\|}$\\[.2mm]
\hline
 & $(0.1) \times 2^{0}$             & 1.9455e-03    & 6.2543e-04 &   &\\
 & $(0.1) \times 2^{-1}$             & 4.8861e-04  &  1.6088e-04  &2.00&1.96  \\
 & $(0.1) \times 2^{-2}$             & 1.2240e-04   & 4.0780e-05 &2.00   &1.98\\
$r=1$  & $(0.1) \times 2^{-3}$             &   3.0629e-05 &1.0264e-05 &  2.00  & 1.99 \\
 & $(0.1) \times 2^{-4}$             &7.6607e-06  & 2.5748e-06 &   2.00 &2.00 \\
 & $(0.1) \times 2^{-5}$             &  1.9156e-06    &  6.4477e-07& 2.00  &2.00 \\
\hline
 & $(0.1) \times 2^{0}$             & 2.6708e-05    &  1.3269e-05 &   &\\
 & $(0.1) \times 2^{-1}$             &  3.3523e-06      &  1.6837e-06 & 2.99  &2.98\\
 & $(0.1) \times 2^{-2}$             & 4.1979e-07   &  2.1202e-07 &  3.00  &2.99\\
$r=2$  & $(0.1) \times 2^{-3}$             & 5.2518e-08    &2.6599e-08  &  3.00  &3.00\\
 & $(0.1) \times 2^{-4}$             &  6.5673e-09  & 3.3308e-09  &  3.00  &3.00 \\
 & $(0.1) \times 2^{-5}$             & 8.2108e-10    & 4.1672e-10& 3.00   & 3.00 \\
\hline 
 & $(0.1) \times 2^{0}$             &  2.8964e-07     & 9.5564e-08 &   &\\
 & $(0.1) \times 2^{-1}$             &  1.8172e-08    & 6.0617e-09  & 4.00   &  3.98\\
 & $(0.1) \times 2^{-2}$             & 1.1377e-09    &  3.8151e-10 &   4.00 & 3.99 \\
$r=3$  & $(0.1) \times 2^{-3}$             &7.1152e-11    &  2.3918e-11   & 4.00   & 4.00 \\
 & $(0.1) \times 2^{-4}$             &  4.4370e-12 & 1.4871e-12&   4.00 & 4.01 \\
 & $(0.1) \times 2^{-5}$             &  2.7555e-13   &  8.4657e-14  &   4.01 & 4.13 \\
\end{tabular}
\end{center}
\label{tb1}
\end{table}

\subsection{Nonlinear problem}

In this part, we consider the following nonlinear optimal control problem:
\[
J = \frac12\int_0^{1/5} x^2(t) + u^2(t)\,dt
\]
subject to the state equation
\bq\label{lin_ex1}
x'(t) = x^2(t) + u(t), \quad x(0) = 2.
\eq
In this case, the corresponding adjoint equation and optimal control are given as follows.
\[
\lambda'(t) = - x(t)(1 + \lambda(t)) \quad \mbox{and} \quad \bar u(t) = - \lambda(t),
\]
and thus the optimal solution $\bar x$ solves
\[
x'(t) = x^2(t) -\lambda(t), \quad x(0) = 2.
\]
In this case, since we have no explicit form of the actual solutions, we take the reference solutions $\bar{x}_h$ (resp., $\bar{u}_h$) with $h= (0.1) \times 2^{-9}$ instead of $\bar{x}$ (resp., $\bar{u}$). In Table \ref{tb3}, we arrange the discrete $L^2$ error between reference solutions and its approximations.
\begin{table}[htp]
\caption{Discrete $L^2$ error: $\|\bar x - \bar x_h\|_{L^2(I)}$ and $\|\bar u - \bar u_h\|_{L^2(I)}$}
\begin{center}
\begin{tabular}{c|ccccc}
 \mbox{\quad} $$  \mbox{\quad}& \mbox{\quad} $h$  \mbox{\quad}&  \mbox{\quad} $\|\bar x - \bar x_h\|_{L^2(I)}$  \mbox{\quad}&  \mbox{\quad} $\|\bar u - \bar u_h\|_{L^2(I)}$ \mbox{\quad}&$\log_2 \frac{\|\bar{x}-\bar{x}_{2h}\|}{\|\bar{x}-\bar{x}_{h}\|}$\mbox{\quad}&$\log_2 \frac{\|\bar{u}-\bar{u}_{2h}\|}{\|\bar{u}-\bar{u}_{h}\|}$\\[.2mm]
\hline
 & $0.1$                              & 1.3006e-02    &2.6587e-03 &      &        \\
 & $(0.1) \times 2^{-1}$             & 4.5715e-03     &  6.8872e-04  &  1.51    &       1.95 \\
 & $(0.1) \times 2^{-2}$             &1.3286e-03   & 1.7024e-04 &     1.78 &    2.02    \\
$r=1$  & $(0.1) \times 2^{-3}$             & 3.5677e-04    & 4.2187e-05 &   1.90  & 2.01       \\
 & $(0.1) \times 2^{-4}$             & 9.2305e-05   & 1.0492e-05 &    1.95  &      2.01 \\
 & $(0.1) \times 2^{-5}$             &   2.3420e-05  & 2.6101e-06&     1.98 &     2.01  \\
\hline
 & $0.1$                              & 7.9288e-04    & 7.1751e-05  &      &        \\
 & $(0.1) \times 2^{-1}$             & 1.6928e-04    & 6.8412e-06 &  2.23    &     3.40   \\
 & $(0.1) \times 2^{-2}$             &  2.7566e-05     &7.2059e-07 &    2.62  &      3.25  \\
$r=2$  & $(0.1) \times 2^{-3}$             &3.9391e-06    &8.4373e-08&  2.81    &    3.10    \\
 & $(0.1) \times 2^{-4}$             &  5.2676e-07   & 1.0332e-08  &     2.90 &     3.03   \\
 & $(0.1) \times 2^{-5}$             &  6.8107e-08  & 1.2833e-09 &   2.95   &     3.01   \\
\hline 
 & $0.1$                              & 4.8978e-05    & 2.3326e-06 &      &        \\
 & $(0.1) \times 2^{-1}$             &  5.8217e-06    &2.0158e-07 &  3.07    &   3.53     \\
 & $(0.1) \times 2^{-2}$             & 5.0236e-07  &    1.3655e-08&   3.53   &      3.88  \\
$r=3$  & $(0.1) \times 2^{-3}$             & 3.6929e-08   &8.7619e-10&  3.77    & 3.96       \\
 & $(0.1) \times 2^{-4}$             &   2.5037e-09  & 5.5551e-11&    3.88  &      3.98  \\
 & $(0.1) \times 2^{-5}$             & 1.6329e-10   & 3.6858e-12  &    3.94  &      3.91  \\
\end{tabular}
\end{center}
\label{tb3}
\end{table}

%
%
%
%

\section*{Acknowledgments}
\noindent \textbf{Acknowledgments}
The work of W. Choi is supported by NRF grant (No. 2017R1C1B5076348) and Faculty Research Fund, Sungkyunkwan University. The work of Y.-P. Choi is supported by NRF grant (No. 2017R1C1B2012918) and Yonsei University Research Fund of 2019-22-021.

%
%
%

\appendix

%
%
%

\section{Differentiability of the control-to-state mapping}\label{app_d}

In this section, we show that the control-tostate mapping $G$ is twice differentiable, and obtain the derivatives.  
\begin{lem}\label{lem-d-1}Let $x^s = G(u+sv)$ and $y:[0,T] \rightarrow \mathbb{R}^d$ be the solution of 
\begin{equation*}
y' (t)= \frac{\partial f}{\partial x} (t, x(t), u(t))y(t) + \frac{\partial f}{\partial u} (t, x(t), u(t)) v(t),~t \in (0,T),\quad y(0) =0.
\end{equation*}
Then we have  
\[
\frac{d}{ds} G(u+sv)|_{s=0} = y.
\]
\end{lem}
\begin{proof}
Recall that $x^s$ and $x$ satisfy
\begin{equation*}
(x^s)' (t)= f(t,x^s (t), u(t) + s v(t)) \quad \textrm{and}\quad x' (t) = f (t, x(t), u(t)),
\end{equation*}
respectively. Using this, we find that $r(t) := x^s (t) - x (t) - s y(t)$ satisfies 
\begin{equation}\label{eq-d-43}
\begin{split}
&(x^s (t) - x (t) - s y(t))' (t)
\\
 & \quad = f(t, x^s, u+s v) - f(t, x, u) - s \left( \frac{\partial f}{\partial x}(t, x, u) y(t) + \frac{\partial f}{\partial u}(t, x, u) v(t) \right) 
\\
& \quad =: \frac{\partial f}{\partial x}(t,x,u) (x^s (t) - x(t) -s y(t))+ A_1(t) + A_2(t),
\end{split}
\end{equation}
where
\begin{equation*}
A_1(t) := f(t, x^s, u) - f(t,x,u) - \frac{\partial f}{\partial x}(t,x,u) (x^s (t) - x(t))
\end{equation*}
and
\begin{equation*}
A_2(t) := f(t,x^s, u+sv) - f(t,x^s, u) - s \frac{\partial f}{\partial u}(t,x,u) v(t).
\end{equation*}
Given that $|x^s (t) - x (t)|\leq Cs$ and \eqref{eq-1-10}, an elementary calculus shows that $|A_1| \leq Cs^2$ and $|A_2| \leq Cs^2$. 
With these bounds, we may apply the Gr\"onwall's lemma for \eqref{eq-d-43} to deduce $|r(t)| \leq Cs^2$ for $t \in [0,T]$. From this we find 
\begin{equation*}
\lim_{s \rightarrow 0} \frac{x^s (t) - x(t) - sy (t)}{s} = 0,
\end{equation*}
which yields that 
\[
\frac{d}{ds} x^s (t) = y(t).
\] 
\end{proof}

Next we show the twice differentiablity of the mapping $s \rightarrow G(u+sv)$ at $s=0$. 
\begin{lem}\label{lem-d-33} Let $z:[0,T] \rightarrow \mathbb{R}^d$ be the solution of 
$$\begin{aligned}
z' (t) &= \frac{\partial^2 f}{(\partial x)^2} (t,x(t),u(t)) y^2(t) + 2 \frac{\partial^2 f}{\partial x \partial u} (t,x(t),u(t)) y(t)v(t) + \frac{\partial^2 f}{(\partial u)^2} (t,x(t),u(t)) v^2(t) 
\\
&\quad + \frac{\partial f}{\partial x}(t,x(t),u(t)) z(t),\quad z(0)=0.
\end{aligned}$$
Then we have 
\[
\frac{d^2}{(ds)^2} G(u+sv)|_{s=0} = z(t).
\]
\end{lem}
\begin{proof}
Let 
\[
y^s (t) = \frac{d}{ds}G(u+sv) \quad \mbox{and} \quad y(t)= \frac{d}{ds} G(u+sv)|_{s=0}.
\] 
Then we get
\begin{equation}\label{eq-d-41}
\begin{split}
&(y^s)' (t) -y' (t) - sz' (t)\\
&\quad = \frac{\partial f}{\partial x} (t,x^s, u+sv) y^s (t) + \frac{\partial f}{\partial u} (t, x^{s}, u+sv) v(t) - \frac{\partial f}{\partial x}(t, x, u) y(t)- \frac{\partial f}{\partial u}(t,x, u) v(t)
\\
&\qquad - s\left[ \frac{\partial^2 f}{(\partial x)^2} (t,x(t), u) y^2 (t) + 2 \frac{\partial^2 f}{\partial x \partial u} (t,x(t), u) y(t) v(t) \right.
\\
&\hspace{3cm} + \left. \frac{\partial^2 f}{(\partial u)^2} (t,x(t), u) v^2 (t) + \frac{\partial f}{\partial x}(t,x(t), u) z(t) \right]
\\
&\quad =: \frac{\partial f}{\partial x}(t,x(t), u) (y^s (t)- y (t) - sz(t)) + A_1 (t) + A_2 (t),
\end{split}
\end{equation}
where
\begin{equation*}
\begin{split}
A_1 (t)&:= \left[\frac{\partial f}{\partial x}(t, x^s, u+sv) - \frac{\partial f}{\partial x}(t,x,u)\right] y^s (t) 
+ s\left[ \frac{\partial^2 f}{(\partial x)^2} (t,x, u) y(t) + \frac{\partial^2 f}{\partial x \partial u}(t,x, u) v(t) \right] y(t),
\end{split}
\end{equation*}
\begin{equation*}
\begin{split}
A_2 (t) &:=\left[\frac{\partial f}{\partial u}(t,x^s, u+sv) - \frac{\partial f}{\partial u}(t,x(t), u) \right] v(t) 
 s \left[ \frac{\partial^2 f}{(\partial u)^2} (t,x, u) v(t) + \frac{\partial^2 f}{\partial x \partial u} (t,x, u) y(t) \right]v(t).
\end{split}
\end{equation*}
By Lemma \ref{lem_e1} we have $|y^s (t) - y(t)|\leq Cs$. Given this estimate and that 
\[
\frac{d}{ds} x^s (t)|_{s=0} = y(t)
\] 
from Lemma \ref{lem-d-1}, an elementary calculus shows that $|A_1 (t)| \leq Cs^2$ and $|A_2 (t)| \leq Cs^2$. Inserting this estimate into \eqref{eq-d-41} and applying the Gr\"onwall's lemma, we find
\begin{equation*}
y^s (t) - y(t) - sz(t) = O(s^2).
\end{equation*}
It proves that 
\[
\frac{d}{ds} y^s (t)|_{s=0} = z(t).
\] 
This implies that 
\[
\frac{d^2}{(ds)^2} G(u+sv)|_{s=0} =z(t)
\] 
since 
\[
y^s (t) = \frac{d}{ds} G(u+sv).
\] 
This completes the proof.
\end{proof}

\section{Gronwall-type inequality for the DG discretization of ODEs}\label{app_c}
In this section, we provide a Gronwall-type inequality for the DG discretization of ODEs with inputs. It will be used in Section \ref{app_e} to establish the differentiability of the discrete control-to-state mapping $G_h$. 

We begin with recalling from \cite[Lemma 2.4]{SS00} the following lemma.
\begin{lem}\label{lem-d-1}Let $I= (a,b)$ and $k = b-a >0$. Then we have
\begin{equation*}
\int_a^b |\phi (t)|^2 \,dt \leq \frac{1}{k} \sum_{i=1}^{d} \left( \int_a^b \phi_i (t) \, dt \right)^2 + \frac{1}{2} \int_a^b (b-t) (t-a) |\phi' (t)|^2 \,dt
\end{equation*}
for all $\phi (t) = (\phi_1 (t), \cdots, \phi_d (t)) \in P^r ((a,b); \mathbb{R}^d)$, $r \in \mathbb{N}_0$, where
\begin{equation*}
P^{r}((a,b); \mathbb{R}^d) = \{ (p_1, \cdots, p_d)~:~ p_k :(a,b) \rightarrow \mathbb{R}~\textrm{is a polynomial of order $\leq r$}\}.
\end{equation*}
\end{lem}
\noindent The next result is from \cite[Lemma 3.1]{SS00}.
\begin{lem}\label{lem-d-2} For $I = (a,b)$ and $r \in \mathbb{N}_0$, we have
\begin{equation*}
\|\phi\|_{L^{\infty}(I)}^2 \leq C \log (r+1) \int_a^b |\phi' (t)|^2 (t-a) \,dt + C|\phi (b)|^2
\end{equation*}
for all $\phi (t) = (\phi_1 (t), \cdots, \phi_d (t)) \in P^r ((a,b); \mathbb{R}^d)$. Here $C>0$ is independent of $r$, $a$, $b$, and $d$. 
\end{lem}
\noindent We shall use the following Gr\"onwall inequality.
\begin{lem}\label{lem-d-3}
 Let $\{a_n\}_{n=1}^{N}$ and $\{b_n\}_{n=1}^{N}$ be sequences of non-negative numbers satisfying $b_1 \leq b_2 \leq\cdots \leq b_{N}$ and $b_1= 0$. Assume that for a value $h \in (0,1/2)$ we have
 \begin{equation*}
 (1-h) b_{n+1} \leq b_{n} + a_{n}
 \end{equation*}
 for $n \in \mathbb{N}$. Then there exists a constant $Q>0$ independent of $h \in (0,1/2)$ and $N \in \mathbb{N}$ such that
 \begin{equation*}
 b_n \leq e^{Q (nh)} \sum_{k=1}^{n} a_k
 \end{equation*} 
 for any $n \in \mathbb{N}$ with $n \leq N/h$.
\end{lem}
\begin{proof} The proof can be obtained by induction.  
\end{proof}
\noindent Now we obtain the Gr\"onwall-type inequality. 
\begin{lem}\label{lem-d-4} Suppose that
\begin{equation}\label{eq-d-r}
\begin{split}
|B(x, \varphi)| \leq C \sum_{n=1}^{N} \lt( |(x(t), \varphi (t))_{I_n}| + |(u(t), \varphi(t))_{I_n}|\rt)
\end{split}
\end{equation}
for all $\varphi \in X_h^r$. Then there exists a constant $C>0$ independent of $h >0$ such that
\begin{equation*}
\| x \|_{L^{\infty}(I)} \leq C \|u \|_{L^{2}(I)}
\end{equation*}
for all $u_1, u_2 \in \mathcal{U}_{ad}$ and $h > 0$ small enough. 
\end{lem}
\begin{proof}
From the condition \eqref{eq-d-r} we have
\begin{equation*}
\begin{split}
&\left| \sum_{n=1}^{N} (x' (t), \varphi (t))_{I_n}  + \sum_{n=2}^{N} ([x]_{n-1}, \varphi_{n-1}^{+})_{I_n} + (x_{0}^{+}, \varphi_0^{+})_{I_1}\right|\cr
&\quad  \leq C \sum_{n=1}^{N}  |(x(t), \varphi (t))_{I_n}| + |(u(t), \varphi(t))_{I_n}|
\end{split}
\end{equation*}
for all $\varphi \in X_h^r$. To obtain the desired estimates, for each $n \in \{1, \cdots, N\}$ we shall take the following test functions $\varphi \in X_h^r$ supported on $I_n$ given as 
\begin{equation*}
\begin{split}
\varphi (t)& = (x_1 -x_2) (t) 1_{I_n} (t),
\\
\varphi (t) &= (t-t_{n-1}) (x_1 - x_2)' (t) 1_{I_n} (t), \quad \mbox{and}
\\
\varphi (t) & = (t-t_{n-1})1_{I_n}(t),
\end{split}
\end{equation*}
where $1_{I_n}: I \to \{0,1\}$ denotes the indicator function, that is, $1_{I_n}(t) = 1$ for $t \in I_n$ and $1_{I_n}(t) = 0$ for $t \in I\setminus I_n$.
First we take $\varphi (t) = x(t)1_{I_n}(t)$ for $n=1,2,\cdots, N$. Then,  
\begin{equation}\label{eq-d-1}
\begin{split}
(x' (t), x (t))_{I_n}+ \left( [x]_{n-1}, x_{n-1}^{+}\right) \leq  C  |(x(t), x (t))_{I_n}| + |(u(t),x(t))_{I_n}|,
\end{split}
\end{equation}
where for $n=1$ we abuse a notation $[x]_{0}$ to mean $x_0^{+}$. Notice that 
\begin{equation*}
([x]_{n-1}, x_{n-1}^{+}) = (x_{n-1}^{+})^2 - (x_{n-1}^{-}, x_{n-1}^{+}),
\end{equation*}
where for $n=1$ the above is understood as $([x]_{0}, x_0^{+}) = (x_0^{+})^2.$ Using this in \eqref{eq-d-1}, we find
\begin{equation*}
\begin{split}
&\frac{1}{2} |x_{n}^{-} |^2 - \frac{1}{2} |x_{n-1}^{+}|^2 +| x_{n-1}^{+}|^2\leq (x_{n-1}^{-}, x_{n-1}^{+}) +  C  |(x(t), x (t))_{I_n}| + |(u(t),x(t))_{I_n}|.
\end{split}
\end{equation*}
By applying Cauchy-Schwarz inequality, we obtain
\begin{equation}\label{eq-d-2}
\frac{1}{2}|x_{n}^{-}|^2 \leq \frac{1}{2} |x_{n-1}^{-}|^2 
+  C  \|x(t)\|_{L^2(I_n)}^2 + C\|u(t)\|_{L^2 (I_n)}^2.
\end{equation}
Secondly, we take $\varphi (t) = (t-t_{n-1})x' (t)1_{I_n} (t)$ to have 
\begin{equation*}
\begin{split}
(x' (t), (t-t_{n-1}) x' (t))_{I_n}  \leq  \left( x(t),~(t-t_{n-1})x' (t) \right)_{I_n} +\left( u(t),~(t-t_{n-1})x' (t) \right)_{I_n}.
\end{split}
\end{equation*}
By using H\"older's inequality, we get
\begin{equation}\label{eq-d-3}
\int_{I_n} (t-t_{n-1}) |x' (t)|^2\,dt \leq \int_{I_n} |t-t_{n-1}| (|x(t)|^2 + |u(t)|^2 )\,dt.
\end{equation}
Notice that
\begin{equation*}
\left(x' (t), (t-t_{n-1})\right)_{I_n} = - \int_{I_n} x(t) \,dt + x (t_n) (t_n - t_{n-1}).
\end{equation*}
Thus, choosing $\varphi (t) = (t-t_{n-1})1_{I_n}(t)$ gives
\begin{equation*}
\begin{split}
\left|\int_{I_n} x(t)\,dt + x (t_n) (t_n - t_{n-1}) \right|
\leq  C \int_{I_n} |x(t)| (t-t_{n-1})\,dt + C \int_{I_n} |u(t)| (t-t_{n-1})\,dt,
\end{split}
\end{equation*}
and subsequently, this yields
\begin{equation*}
\begin{split}
\left| \int_{I_n} x(t)\,dt \right|^2 &  \leq 2h_n^2 | x_{n}^{-}|^2 + 2 \int_{I_n}( x(t)^2+ u(t)^2 )\, dt \int_{I_n} (t_{n-1} -t)^2 \,dt
\\
&  \leq 2h_n^2 |x_{n}^{-}|^2 + Ch_n^3  \int_{I_n}( x(t)^2+ u(t)^2 )\, dt,
\end{split}
\end{equation*}
where $h_n = t_n - t_{n-1}$.
This together with Lemma \ref{lem-d-1} asserts
\begin{equation}\label{eq-d-5}
\begin{split}
\left| \int_{I_n} x (t) \,dt\right|^2 & \leq 2h_n^2 |x_n^{-}|^2 + C h_n^4 \int_{I_n} (t-t_{n-1}) |x' (t)|^2 \,dt  + Ch_n^3 \int_{I_n}  |u(t)|^2 \,dt
\end{split}
\end{equation}
for $h > 0$ small enough. Combining \eqref{eq-d-2} and \eqref{eq-d-3}, we find
\begin{equation*}
\begin{split}
&\int_{I_n} (t-t_{n-1}) |x' (t)|^2 \,dt + |x_{n}^{-}|^2
\\
&\quad \leq C\|x\|_{L^2(I_n)}^2 + C \int_{I_n}|u(t)|^2 \,dt + |x_{n-1}^{-}|^2
\\
&\quad \leq \frac{C}{h_n} \left| \int_{I_n} x(t)\, dt\right|^2 + Ch_n \int_{I_n} (t-t_{n-1}) |x' (t)|^2 \,dt + |x_{n-1}^{-}|^2 + C\int_{I_n} |u(t)|^2\, dt,
\end{split}
\end{equation*}
where we applied Lemma \ref{lem-d-1} in the second inequality. This, together with \eqref{eq-d-5}, we obtain
\begin{equation}\label{eq-t-1}
\frac{1}{2}\int_{I_n} (t-t_{n-1}) |x' (t)|^2 \,dt + |x_{n}^{-}|^2
 \leq Ch_n |x_{n}^{-}|^2 + |x_{n-1}^{-}|^2 +  C\int_{I_n} |u(t)|^2 \, dt
\end{equation}
for $h > 0$ small enough, where for $n=1$ one has $|x_0^{-}| =0$. This inequality trivially gives 
\begin{equation*}
 |x_{n}^{-}|^2
 \leq Ch_n |x_{n}^{-}|^2 + |x_{n-1}^{-}|^2 +  C\int_{I_n} |u(t)|^2 \, dt
\end{equation*}
for $n=1, \cdots, N$. 
Now, by applying Lemma \ref{lem-d-3} to find an estimate of $|x_n^{-1}|^2$ and inserting it into \eqref{eq-t-1}, we achieve
\begin{equation*}
\frac{1}{2}\int_{I_n} (t-t_{n-1}) |x' (t)|^2 \,dt + |x_{n}^{-}|^2 \leq C \int_{0}^{T} |u(t)|^2 \, dt.
\end{equation*}
Finally, by applying Lemma \ref{lem-d-2} to the above, we obtain the desired estimate.
\end{proof}
As a corollary, we have the following Lipshitz estimates.
\begin{lem}\label{lem-3-5} For $u, v \in \mathcal{U}_{ad}$ we have
\begin{equation*}
\|G_h (u) - G_h (v)\|_{L^{\infty}(I)} \leq C \|u-v\|_{L^2 (I)}
\end{equation*}
and
\begin{equation*}
\|\lambda_h (u) - \lambda_h (v) \|_{L^{\infty}(I)} \leq C \|u-v\|_{L^2 (I)}.
\end{equation*}
\end{lem}
\begin{proof}
Let us denote by $x = G_h(u)$ and $\hat{x} = G_h(v)$. Then it follows from \eqref{eq-dg} that
$$\begin{aligned}
B((x -\hat{x}), \varphi )&= \Big(f(t,x(t),u(t)) -  f(t,\hat x (t), \hat u (t)), ~\varphi\Big) \quad \forall\,\varphi \in X_h^r.
\end{aligned}$$
By \eqref{eq-1-10}, there exists a constant $C>0$ such that
\begin{equation*}
\begin{split}
\left|f(t,x(t),u(t)) - f(t,\hat x(t), \hat{u}(t))\right| &\leq C|\hat{x}(t) - x(t)| + C |\hat{u}(t) - u(t)|.
\end{split}
\end{equation*}
By applying Lemma \ref{lem-d-4}, we get the inequality
\begin{equation*}
\|x - \hat{x}\|_{L^{\infty}(I)} \leq  C \|u-\hat{u}\|_{L^{2}(I)} .
\end{equation*}
This gives the first inequality. For the second one, we denote by $\lambda = \lambda_h (u)$ and $\hat \lambda = \lambda_h (v)$. Then, we see from Lemma \ref{eq-2-1} that
$$\begin{aligned}
B(\varphi,~(\lambda - \hat\lambda)) &= \Big(\varphi, ~\pa_x f(\cdot,x,u)(\lambda - \hat\lambda)(t) + (\pa_x f(\cdot,x,u) - \pa_x f(\cdot,\hat x , \hat u))(t)
\cr
&\hspace{2cm} -  (\pa_x g(\cdot,x,u) - \pa_x g(\cdot, \hat x, \hat u))\Big)_{I} \quad \forall\,\varphi \in X_h^r.
\end{aligned}$$
By applying Lemma \ref{lem-d-4} again in a backward way (see Lemma \ref{lem-3-6}), we obtain
$$\begin{aligned}
\|\lambda - \hat\lambda\|_{L^\infty(I)} &\leq C\|(\pa_x f(\cdot,x,u) - \pa_x f(\cdot,\hat x , \hat u))\hat\lambda\|_{L^2(I)} \cr
&\quad + C\|\pa_x g(\cdot,x,u) - \pa_x g(\cdot, \hat x, \hat u)\|_{L^2(I)}\cr
&\leq C(\|\hat \lambda\|_{L^\infty(I)} + 1) \lt( \|x - \hat x\|_{L^\infty(I)} + \|u - \hat u\|_{L^2(I)} \rt)\cr
&\leq C\|u - \hat u\|_{L^2(I)},
\end{aligned}$$
where we used 
\[
\|\hat \lambda\|_{L^\infty(I)} \leq C\|\pa_x g\|_{L^\infty (I)}
\]
due to Lemma \ref{lem-d-4}. This completes the proof.
\end{proof}

\section{Differentiability of discrete control-to-state mapping}\label{app_e}

This section is devoted to prove that the discrete control-tostate mapping $G_h$ is twice differentiable. We also obtain the first and second derivatives of $G_h$.  
\begin{thm}\label{thm-4-2} We denote $x_h^s = G_h(u+sv)$ and set $y_h \in X_h^r$ be the solution of the following discretized equation:
\bq\label{eqn_yh}
B(y_h, \varphi ) = \left( \frac{\partial f}{\partial x} (t,x_h, u) y_h (t) + \frac{\partial f}{\partial u}(t, x_h, u) v(t),\,\varphi (t)\right)_{I} \quad \forall \,\varphi \in X_h^r,
\eq
where $x_h = G_h (u)$. Then we have $\frac{d}{ds} x_h^s (t) = y_h(t)$.
\end{thm}
\begin{proof}
By Theorem \ref{thm-21} there exists a solution $y_h \in X_h^r$ to
\[
B(y_h, \varphi ) = \left( \frac{\partial f}{\partial x} (t,x_h, u) y_h (t) + \frac{\partial f}{\partial u}(t, x_h, u) v(t),\,\varphi (t)\right)_{I} \quad \forall \,\varphi \in X_h^r.
\]
By Lemma \ref{lem-d-4} we get
\begin{equation}\label{eq-d-61}
\|y_h\|_{L^{\infty}(I)} \leq C\|v\|_{L^2 (I)}.
\end{equation}
Recall that $x^s$ and $x$ satisfy
\begin{equation*}
B(x_h^s, \varphi )= \Big(f(t,x^s, u + s v), ~\varphi (t)\Big)_{I} \quad \textrm{and}\quad B(x_h (t),~\varphi ) = \Big( f (t, x, u),~\varphi (t) \Big)_{I}.
\end{equation*}
Using this, we find that $r(t) := x_h^s (t) - x_h (t) - s y_h(t)$ satisfies 
\begin{equation}\label{eq-d-43}
\begin{split}
&B\Big((x_h^s - x_h - s y_h), ~\varphi \Big)
\\
 &\quad = \left(f(t, x_h^s, u+s v) - f(t, x_h, u) - s \left( \frac{\partial f}{\partial x}(t, x_h, u) y(t) + \frac{\partial f}{\partial u}(t, x_h, u) v(t) \right),~\varphi (t)\right) 
\\
&\quad = \left(\frac{\partial f}{\partial x}(t,x_h,u) (x_h^s (t) - x_h(t) -s y_h(t))+ A_1 + A_2,~\varphi (t) \right)
\end{split}
\end{equation}
for all $\varphi \in X_h^r$, where
\begin{equation*}
A_1 = f(t, x_h^s, u) - f(t,x_h,u) - \frac{\partial f}{\partial x}(t,x_h,u) (x_h^s (t) - x_h(t))
\end{equation*}
and
\begin{equation*}
A_2 = f(t,x_h^s, u+sv) - f(t,x_h^s, u) - s \frac{\partial f}{\partial u}(t,x_h,u) v(t).
\end{equation*}
Given that $|x_h^s (t) - x_h (t)|\leq Cs$ and \eqref{eq-1-10}, an elementary calculus shows that $|A_1| \leq Cs^2$ and $|A_2| \leq Cs^2$. 
With these bounds, we may apply Lemma \ref{lem-d-4} to deduce $|r(t)| \leq Cs^2$ for $t \in [0,T]$. From this we find that
\begin{equation*}
\lim_{s \rightarrow 0} \frac{x_h^s (t) - x_h(t) - sy_h (t)}{s} = 0,
\end{equation*}
which yields that 
\[
\frac{d}{ds} x_h^s (t) = y_h(t).
\] 
This completes the proof.
\end{proof}

\begin{lem}\label{lem-d-32} The following holds.
\[
\|{G_h}' (u_1) v - {G_h}' (u_2) v\|_{L^{\infty}(I)} \leq C \|u_1 - u_2\|_{L^{2}(I)}\|v\|_{L^{\infty}(I)}.
\]
\end{lem}
\begin{proof}
Let $y_h = {G_h}' (u_1) v \in X_h^r$ and $z_h = {G_h}' (u_2)v \in X_h^r$. Then we obtain
\[
B(y_h, \varphi ) = \left( \frac{\partial f}{\partial x} (t,G_h (u_1), u_1) y_h (t) + \frac{\partial f}{\partial u}(t, G_h (u_1), u_1) v(t),\,\varphi (t)\right)_{I}
\]
and
\[
B(z_h, \varphi ) = \left( \frac{\partial f}{\partial x} (t,G_h (u_2), u_2) z_h (t) + \frac{\partial f}{\partial u}(t, G_h (u_2), u_2) v(t),\,\varphi (t)\right)_{I}
\]
for all $\varphi \in X_h^r$.
Combining these equalities, we have
\begin{equation}\label{eq-d-70}
\begin{split}
B(y_h -z_h, \varphi ) &= \left( \frac{\partial f}{\partial x} (t,G_h (u_1), u_1) (y_h -z_h) (t),\,\varphi (t)\right)_{I}\\
&\quad + \left(\lt( \frac{\partial f}{\partial x} (t,G_h (u_1), u_1) - \frac{\partial f}{\partial x}(t, G_h (u_2), u_2) \rt) z_h (t),\,\varphi (t)\right)_{I}\\
&\quad + \left( \lt(\frac{\partial f}{\partial u} (t,G_h (u_1), u_1) -\frac{\partial f}{\partial u}(t, G_h (u_2), u_2)\rt) v(t),\,\varphi (t)\right)_{I}
\end{split}
\end{equation}
for all $\varphi \in X_h^r$. On the other hand, the following two inequalities hold:
\[
\begin{split}
&\lt|\lt( \frac{\partial f}{\partial x} (t,G_h (u_1), u_1) - \frac{\partial f}{\partial x}(t, G_h (u_2), u_2) \rt) z_h (t)\rt|
 \leq C( |u_1 - u_2| + |G_h (u_1) - G_h (u_2)| ) |z_h (t)| 
\end{split}
\]
and
\[
\begin{split}
&\lt|\lt(\frac{\partial f}{\partial u} (t,G_h (u_1), u_1) -\frac{\partial f}{\partial u}(t, G_h (u_2), u_2)\rt) v(t)\rt|
\leq  C( |u_1 - u_2| + |G_h (u_1) - G_h (u_2)| ) |v (t)|. 
\end{split}
\]
Given these estimates, by applying Lemma \ref{lem-d-4} to \eqref{eq-d-70}, we obtain
\[
\begin{split}
\|y_h -z_h\|_{L^{\infty}(I)}& \leq  C\big\|( |u_1 - u_2| + |G_h (u_1) - G_h (u_2)| ) |z_h (t)|\big\|_{L^2(I)} \cr
&\quad +   C\big\|( |u_1 - u_2| + |G_h (u_1) - G_h (u_2)| ) |v (t)|\big\|_{L^2(I)}
\\
& \leq C \| u_1 - u_2\|_{L^2(I)} \|v\|_{L^{\infty}(I)},
\end{split}
\]
where we used Lemma \ref{lem-3-5} in the second inequality.
\end{proof}

\begin{lem}\label{lem-e-3}
Let $z_h \in X^r_h$ be the solution of the following discretized equation:
$$\begin{aligned}
&B(z_h, \varphi) \cr
& \ = \int_0^{T}\left(\frac{\partial^2 f}{(\partial x)^2} (t,x_h,u) y_h^2(t) + 2 \frac{\partial^2 f}{\partial x \partial u} (t,x_h,u) y_h(t) v(t) + \frac{\partial^2 f}{(\partial u)^2} (t,x_h,u) v^2 (t)\right) \varphi(t) \, dt\\
& \ \quad + \int_0^{T} \frac{\partial f}{\partial x}(t,x_h,u) z_h (t) \varphi(t) \,dt
\end{aligned}$$
for any $\varphi \in X_h^r$, where $y_h \in X^r_h$ is the solution of \eqref{eqn_yh}. 
Then we have 
\begin{equation*}
\frac{d^2}{(ds)^2} G_h(u+sv)|_{s=0} = z_h(t).
\end{equation*}
\end{lem}
\begin{proof}
Let 
\[
y_h^s (t) = \frac{d}{ds}G_h(u+sv) \quad \mbox{and} \quad y_h(t)= \frac{d}{ds} G_h(u+sv)|_{s=0}.
\] 
It then follows that
\begin{equation}\label{eq-d-41}
\begin{split}
&B\Big((y_h^s)' (t) -y_h' (t) - sz_h' (t), ~\varphi (t)\Big)
\\
&\quad =: \Big(\frac{\partial f}{\partial x}(t,x_h(t), u) (y_h^s (t)- y_h (t) - sz_h(t)) + A_1 (t) + A_2 (t),~\varphi (t) \Big)
\end{split}
\end{equation}
where
\begin{equation*}
\begin{split}
A_1 (t)&:= \left[\frac{\partial f}{\partial x}(t, x_h^s, u+sv) - \frac{\partial f}{\partial x}(t,x_h,u)\right] y_h^s (t) 
\\
&\quad + s\left[ \frac{\partial^2 f}{(\partial x)^2} (t,x_h, u) y_h(t) + \frac{\partial^2 f}{\partial x \partial u}(t,x_h, u) v(t) \right] y_h(t),
\end{split}
\end{equation*}
\begin{equation*}
\begin{split}
A_2 (t)&:=\left[\frac{\partial f}{\partial u}(t,x_h^s, u+sv) - \frac{\partial f}{\partial u}(t,x_h, u) \right] v(t) 
\\
&\quad - s \left[ \frac{\partial^2 f}{(\partial u)^2} (t,x_h, u) v(t) + \frac{\partial^2 f}{\partial x \partial u} (t,x_h, u) y_h(t) \right]v(t).
\end{split}
\end{equation*}
We obtain from Lemma \ref{lem-d-32} the estimate $|y_h^s (t) - y_h(t)|\leq Cs$. Upon this estimate and that $\frac{d}{ds} x_h^s (t)|_{s=0} = y_h(t)$ from Lemma \ref{thm-4-2}, an elementary calculus reveals that $|A_1 (t)| \leq Cs^2$ and $|A_2 (t)| \leq Cs^2$. Putting this estimate into \eqref{eq-d-41} and using Lemma \ref{lem-d-4}, we find 
\begin{equation*}
y^s (t) - y(t) - sz(t) = O(s^2).
\end{equation*}
This yields that 
\[
\frac{d}{ds} y_h^s (t)|_{s=0} = z_h(t),
\] 
and so we have 
\[
\frac{d^2}{(ds)^2} G_h(u+sv)|_{s=0} =z_h(t)
\] 
since 
\[
y_h^s (t) = \frac{d}{ds} G_h(u+sv).
\] 
The proof is done.
\end{proof}

\section{Derivations of the first order derivative of cost functionals}\label{app_a}

In this part, we give the proofs of Lemma \ref{lem-30} and Lemma \ref{lem-35}. Before presenting it, we shall explain how to derive the discrete adjoint equation \eqref{eq-2-1} from the Lagrangian associated to \eqref{main_eq4}.

Let us first write the Lagrangian of the problem \eqref{main_eq1} and \eqref{dg_b} as follows:
\begin{align}\label{dis_lag}
\begin{aligned}
\mathcal{L}_h({x_h},u,\lambda) &:=  \int_0^T g(t,x_h(t),u(t))\,dt + B(x_h,\lambda_h) - \lt(f(\cdot,x_h,u),\lambda_h \rt)_I - (x_0, \lambda^+_{h,0}),
\end{aligned}
\end{align}
where the bilinear operator $B(\cdot,\cdot)$ is given by \eqref{dg_b}. If we compute the functional derivatives of the above Lagrangian \eqref{dis_lag} with respect to the adjoint state $\lambda_h$, then $\delta\ml_h/\delta \lambda_h = 0$ leads \eqref{dg_b}. We now derive the equation of discrete adjoint state. Using the integration by parts, we find
\[
B(x_h, \lambda_h) = - \sum_{n=1}^N (x_h, \lambda_h')_{I_n} - \sum_{n=1}^{N-1} (x_{h,n}^-, [\lambda_h]_n) + (x_{h,N}^-, \lambda_{h,N}^-).
\]
This enables us to rewrite the Lagrangian \eqref{dis_lag} as
$$\begin{aligned}
\mathcal{L}_h(x,u,\lambda) &=  \int_0^T g(t,x_h(t),u(t))\,dt - \sum_{n=1}^N (x_h, \lambda_h')_{I_n} - \lt(f(\cdot,x_h,u),\lambda_h \rt)_I  \cr
&\quad  - \sum_{n=1}^{N-1} (x_{h,n}^-, [\lambda_h]_n) + (x_{h,N}^-, \lambda_{h,N}^-) - (x_0, \lambda^+_{h,0}),
\end{aligned}$$
and this further implies
\begin{align}\label{dis_adj}
\begin{aligned}
0&=\frac{\delta\ml_h(x,u,\lambda)}{\delta x_h}(\psi_h)\cr 
&=\int_0^T \frac{\pa g}{\pa x}(t,x_h(t),u(t)) \psi_h(t) \,dt - \sum_{n=1}^N  (\psi_h, \lambda_h')_{I_n} - \lt(\frac{\pa f}{\pa x}(\cdot,x_h,u)\psi_h,\lambda_h \rt)_I  \cr
&\quad - \sum_{n=1}^{N-1} (\psi_{h,n}^-, [\lambda_h]_n) + (\psi_{h,N}^-, \lambda_{h,N}^-)
\\
&= \int_0^T \frac{\pa g}{\pa x}(t,x_h(t),u(t)) \psi_h(t) \,dt -   \lt(\frac{\pa f}{\pa x}(\cdot,x_h,u)\psi_h,\lambda_h \rt)_I + B(\psi_h, \lambda_h)
\end{aligned}
\end{align}
for all $\psi_h \in X^r_h$, where we applied the integraion by parts for $(\psi_h, {\lambda_h}')_{I_n}$ to derive the second equality. The above equality corresponds to the adjoint equation \eqref{eq-2-1}.

%
%
%

\begin{proof}[Proof of Lemma \ref{lem-30}]
In order to compute the functional derivative of $j$ with respect to $u$, we consider $j(u+sv) = J(u + sv, G(u+sv))$ with $v \in \U$ and $s \in \R_+$. If we set $x^s(t) := G(u(t)+sv(t))$ it follows from Lemma \ref{lem-d-1} that $y = \frac{d}{ds} x^s (t)|_{s=0}$ satisfies
\begin{equation}\label{eq-b1}
y' (t) = \frac{\partial f}{\partial x} (t,x, u)y(t) + \frac{\partial f}{\partial u} (t,x,u) v(t),
\end{equation}
with the initial condition $y(0)=0$. Recall from \eqref{eq-1-4} that the adjoint state $\lambda (t) = \lambda(u)(t)$ satisfies
\begin{equation}\label{app_y1}
\lambda' (t) = \frac{\partial g}{\partial x}(t,x,u) - \lambda (t) \frac{\partial f}{\partial x}(t,x,u).
\end{equation}
Then we have
$$\begin{aligned}
j'(u)v &= \frac{d}{ds}j(u + sv)\bigg|_{s=0}\cr
&= \int_0^T \frac{\pa g}{\pa u} (t,x(t), u(t))v(t)\,dt + \int_0^T\frac{\pa g}{\pa x} (t,x(t), u(t))y(t)\,dt \cr
&= \int_0^T \lt(\frac{\pa g}{\pa u} (t,x(t), u(t)) - \lambda (t) \frac{\partial f}{\partial u}(t,x(t),u(t)) \rt)v(t)\,dt,
\end{aligned}$$
where we used
$$\begin{aligned}
\int_0^T\frac{\pa g}{\pa x} (t,x(t), u(t))y(t)\,dt &= \int_0^T \lt(\lambda' (t) + \lambda (t) \frac{\partial f}{\partial x}(t,x(t),u(t)) \rt) y(t)\,dt\cr
&=-\int_0^T \lambda (t) \frac{\partial f}{\partial u}(t,x(t),u(t)) v(t)\,dt,
\end{aligned}$$
due to \eqref{eq-b1}, \eqref{app_y1}, $y(0) = 0$, and $\lambda(T) = 0$.
\end{proof}

\begin{proof}[Proof of Lemma \ref{lem-35}]
The proof is very similar to Lemma \ref{lem-30}. We consider $j_h(u+sv) = J(u+sv, G_h(u + sv))$ with $v \in \U$ and $s \in \R_+$. We recall from Lemma \ref{thm-4-2} that the function $x_{h}^s:= G_h(u + sv)$ is differentiable at $s=0$ with 
\[
\frac{d}{ds} x_h^s |_{s=0} = y_h,
\] 
where $y_h \in X_h^r$ satisfies the following equation:
\bq\label{eq_w}
B(y_h,\varphi) = \lt(\frac{\pa f}{\pa x}(\cdot,x_h,u) y_h + \frac{\pa f}{\pa u} (\cdot,x_h, u)v,\varphi \rt)_I \quad \forall\, \varphi \in X^r_h.
\eq
Using this, we obtain
\begin{align}\label{eqn_jk}
\begin{aligned}
j'_h(u)v &= \frac{d}{ds}j_h(u+sv)\bigg|_{s=0} \cr
&= \int_0^T \frac{\pa g}{\pa u} (t,x_h(t), u(t))v(t)\,dt + \int_0^T \frac{\pa g}{\pa x} (t,x_h(t), u(t))y_h(t)\,dt.
\end{aligned}
\end{align}
We then take $\psi_h = y_h$ in \eqref{dis_adj} to get
$$\begin{aligned}
&\int_0^T \frac{\pa g}{\pa x} (t,x_h(t), u(t))y_h(t)\,dt \cr
&\quad = \sum_{n=1}^N (y_h, \lambda_h')_{I_n} + \lt(\frac{\pa f}{\pa x}(\cdot,x_h,u)y_h,\lambda_h \rt)_I  + \sum_{n=1}^{N-1} (y_{h,n}^-, [\lambda_h]_n) - (y_{h,N}^-, \lambda_{k,N}^-).
\end{aligned}$$
On the other hand, by using the integration by parts, we find
$$\begin{aligned}
&\sum_{n=1}^N (y_h, \lambda_h')_{I_n} + \sum_{n=1}^{N-1} (y_{h,n}^-, [\lambda_h]_n) - (y_{h,N}^-, \lambda_{h,N}^-) \cr
&\quad = - \sum_{n=1}^N (y_h', \lambda_h)_{I_n} - \sum_{n=2}^N ([y_h]_{n-1}, \lambda_{h,n-1}^+) - (y_{h,0}^+, \lambda_{h,0}^+) \cr
&\quad = - B(w_h, \lambda_h),
\end{aligned}$$
where $B(\cdot,\cdot)$ is appeared in \eqref{def_B}. This yields
$$\begin{aligned}
\int_0^T \frac{\pa g}{\pa x} (t,x_h(t), u(t))y_h(t)\,dt &= -B(y_h, \lambda_h) +  \lt(\frac{\pa f}{\pa x}(\cdot,x_h,u)y_h,\lambda_h \rt)_I\cr
&=-\lt(\frac{\pa f}{\pa u} (\cdot,x_h, u)v, \lambda_h \rt)_I
\end{aligned}$$
due to \eqref{eq_w}. This together with \eqref{eqn_jk} concludes
\[
j_h'(u)v = \int_0^T \lt(\frac{\pa g}{\pa u} (t,x_h(t), u(t)) - \frac{\pa f}{\pa u}(t,x_h(t), u(t))\lambda_h(t)\rt) v(t)\,dt,
\]
where $v \in \U$.
\end{proof}

%
%
%

\section{Derivations of the second order derivative of cost functionals}\label{app_b}

In this appendix, we provide details of the derivation of the second order derivative of cost functional $j$ and its  discrete version $j_h$. 

\begin{lem}\label{lem_2j} Let $j$ be the cost functional for the optimal control problem \eqref{main_eq1}-\eqref{main_eq2}. Then, for $u \in \U_{ad}$ and $v \in \U$, we have
$$\begin{aligned}
j''(u)(v,v) & = - \int_0^{T} \lambda(t)\lt(\frac{\partial^2 f}{(\partial x)^2} (t,x(t),u(t)) y^2(t) + 2 \frac{\partial^2 f}{\partial x \partial u} (t,x(t),u(t)) y(t)v(t)\rt) dt \cr
&\quad  - \int_0^{T} \lambda(t) \frac{\partial^2 f}{(\partial u)^2} (t,x(t),u(t)) v^2(t)\,dt + \int_0^T \frac{\partial^2 g}{(\partial x)^2} (t,x(t),u(t)) y^2(t)\,dt
\\
&\quad + \int_0^{T}  2 \frac{\partial^2 g}{\partial x \partial u} (t,x,u) y(t) v(t) \,dt + \int_0^T \frac{\partial^2 g}{(\partial u)^2} (t,x(t),u(t)) v^2 (t) \,dt.
\end{aligned}$$

\end{lem}
\begin{proof}
Similarly as in Appendix \ref{app_a}, we consider $j(u+sv) = J(u + sv, G(u+sv))$ with $v \in \U$ and $s \in \R_+$ and set $x^s(t) := G(u(t)+sv(t))$. By Lemma \ref{lem-d-1} and Lemma \ref{lem-d-33} it follows that 
\[
\frac{d}{ds}x^s|_{s=0} = y \quad \mbox{and} \quad \frac{d^2}{(ds)^2}x^s|_{s=0}= z,
\] 
where $y \in X$ is given as in \eqref{eq-b1} and  $z \in X$ is the solution to
$$\begin{aligned}
z' (t) &= \frac{\partial^2 f}{(\partial x)^2} (t,x(t),u(t)) y^2(t) + 2 \frac{\partial^2 f}{\partial x \partial u} (t,x(t),u(t)) y(t)v(t) + \frac{\partial^2 f}{(\partial u)^2} (t,x(t),u(t)) v^2(t) 
\\
&\quad + \frac{\partial f}{\partial x}(t,x(t),u(t)) z(t),
\end{aligned}$$
with the initial condition $z(0) = 0$. Then we obtain
\begin{align}\label{eqn_j2}
\begin{aligned}
j''(u)(v,v) &= \frac{d^2}{ds^2} j(u+sv)\bigg|_{s=0}\cr
&=\frac{d^2}{ds^2} \int_0^{T} g(t,x^s(t), u(t)+sv(t)) \,dt \bigg|_{s=0}
\\
& = \int_0^{T} \frac{\partial g}{\partial x} (t,x(t),u(t)) z(t) \,dt + \int_0^T \frac{\partial^2 g}{(\partial x)^2} (t,x(t),u(t)) y^2(t)\,dt
\\
&\quad + \int_0^{T}  2 \frac{\partial^2 g}{\partial x \partial u} (t,x,u) y(t) v(t) \,dt + \int_0^T \frac{\partial^2 g}{(\partial u)^2} (t,x(t),u(t)) v^2 (t) \,dt.
\end{aligned}
\end{align}
On the other hand, we use \eqref{app_y1} to get
$$\begin{aligned}
&\int_0^{T} \frac{\partial g}{\partial x}(t,x(t),u(t)) z(t) \,dt
\\
&\quad =  \int_{0}^{T} \lambda' (t) z(t) \,dt + \int_0^{T} \frac{\partial f}{\partial x}(t,x(t),u(t)) \lambda (t) z(t) \,dt
\\
&\quad = -\int_0^{T} \lambda(t) z' (t) \,dt+\int_0^{T} \frac{\partial f}{\partial x}(t,x(t),u(t)) \lambda (t) z(t) \,dt
\\
&\quad =- \int_0^{T} \lambda(t)\lt(\frac{\partial^2 f}{(\partial x)^2} (t,x(t),u(t)) y^2(t) + 2 \frac{\partial^2 f}{\partial x \partial u} (t,x(t),u(t)) y(t)v(t)\rt) dt \cr
&\qquad  - \int_0^{T} \lambda(t) \frac{\partial^2 f}{(\partial u)^2} (t,x(t),u(t)) v^2(t) \,dt, 
\end{aligned}$$
where we used $\lambda(T) = 0$ and $z(0) =0$. By combining the above with \eqref{eqn_j2}, we have
$$\begin{aligned}
j''(u)(v,v) & = - \int_0^{T} \lambda(t)\lt(\frac{\partial^2 f}{(\partial x)^2} (t,x(t),u(t)) y^2(t) + 2 \frac{\partial^2 f}{\partial x \partial u} (t,x(t),u(t)) y(t)v(t)\rt) dt \cr
&\quad  - \int_0^{T} \lambda(t) \frac{\partial^2 f}{(\partial u)^2} (t,x(t),u(t)) v^2(t) dt + \int_0^T \frac{\partial^2 g}{(\partial x)^2} (t,x(t),u(t)) y^2(t)\,dt \\
&\quad + \int_0^{T}  2 \frac{\partial^2 g}{\partial x \partial u} (t,x,u) y(t) v(t) \,dt + \int_0^T \frac{\partial^2 g}{(\partial u)^2} (t,x(t),u(t)) v^2 (t) \,dt.
\end{aligned}$$
This completes the proof.
\end{proof}
Next we proceed the similar calculation for the approximate solution. 
\begin{lem}\label{lem_2jk} Let $j_h$ be the discrete cost functional for the optimal control problem \eqref{main_eq1}-\eqref{main_eq2}. Then, for $u \in \U_{ad}$ and $v \in \U$, we have
$$\begin{aligned}
&j''_h (u) (v,v)\cr
&\, =  -\int_0^{T}\left(\frac{\partial^2 f}{(\partial x)^2} (t,x_h,u) y_h^2(t) + 2 \frac{\partial^2 f}{\partial x \partial u} (t,x_h,u) y_h(t) v(t) + \frac{\partial^2 f}{(\partial u)^2} (t,x_h,u) v^2 (t)\right) \lambda_h(t) \, dt
\\
&\, \quad + \int_0^{T}  \lt(\frac{\partial^2 g}{(\partial x)^2} (t,x_h,u) y_h^2 (t) + 2 \frac{\partial^2 g}{\partial x \partial u} (t,x_h,u) y_h (t) v(t) + \frac{\partial^2 g}{(\partial u)^2} (t,x_h,u) v^2(t)\rt) dt.
\end{aligned}$$

\end{lem}
\begin{proof} Similarly as in the proof of Lemma \ref{lem-35}, we consider $j_h(u+sv) = J(u+sv, G_h(u + sv))$ with $v \in \U$ and $s \in \R_+$ and set $x_{h}^s:= G_h(u + sv)$. We recall from  Theorem \ref{thm-4-2} and Theorem \ref{lem-e-3} that 
\[
\frac{d}{ds} x_h^s|_{s=0}  = y_h \quad \mbox{and} \quad \frac{d^2}{(ds)^2}x_h^s|_{s=0} =z_h,
\] 
where $z_h \in X^r_h$ satisfies
$$\begin{aligned}
&B(z_h, \varphi) \cr
& \ = \int_0^{T}\left(\frac{\partial^2 f}{(\partial x)^2} (t,x_h,u) y_h^2(t) + 2 \frac{\partial^2 f}{\partial x \partial u} (t,x_h,u) y_h(t) v(t) + \frac{\partial^2 f}{(\partial u)^2} (t,x_h,u) v^2 (t)\right) \varphi(t) \, dt\\
& \ \quad + \int_0^{T} \frac{\partial f}{\partial x}(t,x_h,u) z_h (t) \varphi(t) \,dt.
\end{aligned}$$
Now a straightforward computation gives
$$\begin{aligned}
j''_h (u) (v,v)&=\frac{d^2}{ds^2} \int_0^{T} g(t,x_h^s(t), u(t)+sv(t)) \,dt\bigg|_{s=0}
\\
& = \int_0^{T} \frac{\partial g}{\partial x} (t,x_h(t),u(t)) z_h(t) \,dt + \int_0^{T} \frac{\partial^2 g}{(\partial x)^2} (t,x_h(t),u(t)) y_h^2 (t) \,dt
\\
&\quad + \int_0^{T}  2 \frac{\partial^2 g}{\partial x \partial u} (t,x_h(t),u(t)) y_h (t) v(t) \,dt + \int_0^{T} \frac{\partial^2 g}{(\partial u)^2} (t,x_h(t),u(t)) v^2 (t) \,dt.
\end{aligned}$$
Note that the discrete adjoint state $\lambda_h(t) = \lambda_h(u)(t)$ satisfies
\[
- B(\psi, \lambda_h ) + \left( \frac{\partial f}{\partial x} (t,x_h, u)\lambda_h, \psi\right)_I = \left( \frac{\partial g}{\partial x}(t, x_h, u), \psi\right)_I
\]
for all $\psi \in X^r_h$. Thus by considering $\psi = z_h \in X^r_h$, we find 
$$\begin{aligned}
 &\left( \frac{\partial g}{\partial x}(t, x_h, u), z_h\right)_I
 \\
&\quad =-B(z_h, \lambda_h ) + \left( \frac{\partial f}{\partial x} (t,x_h, u)\lambda_h, z_h\right)_I
\\
&\quad = -\int_0^{T}\left( \frac{\partial^2 f}{(\partial x)^2} (t,x_h,u) y_h^2(t) + 2 \frac{\partial^2 f}{\partial x \partial u} (t,x_h,u) y_h(t) v(t) + \frac{\partial^2 f}{(\partial u)^2} (t,x_h,u) v^2 (t)\right) \lambda_h(t) \, dt.
\end{aligned}$$
Combining the above equalities, we have
$$\begin{aligned}
&j''_h (u) (v,v)\cr
&\ =  -\int_0^{T}\left(\frac{\partial^2 f}{(\partial x)^2} (t,x_h,u) y_h^2(t) + 2 \frac{\partial^2 f}{\partial x \partial u} (t,x_h,u) y_h(t) v(t) + \frac{\partial^2 f}{(\partial u)^2} (t,x_h,u) v^2 (t)\right) \lambda_h(t) \, dt
\\
&\ \quad + \int_0^{T} \lt(\frac{\partial^2 g}{(\partial x)^2} (t,x_h,u) y_h^2 (t) + 2 \frac{\partial^2 g}{\partial x \partial u} (t,x_h,u) y_h (t) v(t) + \frac{\partial^2 g}{(\partial u)^2} (t,x_h,u) v^2(t)\rt) dt.
\end{aligned}$$
This completes the proof.
\end{proof}

%
%


\begin{thebibliography}{20}
\bibitem{Alt84} W. Alt, On the approximation of infinite optimization problems with an application to optimal control problems, Appl. Math. Optim., 12, (1984), 15--27.
\bibitem{AFS18} W. Alt, U. Felgenhauer, and M. Seydenschwanz, Euler discretization for a class of nonlinear optimal control problems with control appearing linearly, Comput. Optim. Appl., 69, (2018), 825--856.
\bibitem{Bac16} M. Baccouch, Analysis of a posteriori error estimates of the discontinuous Galerkin method for nonlinear ordinary differential equations, Appl. Numer. Math., 106, (2016), 129--153.
\bibitem{DHT81} M. Delfour, W. Hager, and F. Trochu, Discontinuous Galerkin methods for ordinary differential equations, Math. Comp., 36, (1981), 455--473.
\bibitem{DH93} A. L. Dontchev and W. W. Hager, Lipschitzian stability in nonlinear control and optimization,  SIAM J. Control Optim., 31, (1993), 569--603.
\bibitem{DH00} A. L. Dontchev and W. W. Hager, The Euler approximation in state constrained optimal control, Math. Comp., 70, (2000), 173--203.
\bibitem{EKR} G. Elnagar, M. A. Kazemi, M. Razzaghi, The pseudospectral Legendre
method for discretizing optimal control problems, IEEE T. Automat. Contr., 40, (1995), 1793--1796.
\bibitem{E95} D. Estep, A posteriori error bounds and global error control for approximation of ordinary differential equations, SIAM J. Numer. Anal., 32, (1995), 1--48.
\bibitem{Fel03} U. Felgenhauer, On stability of bang-bang type controls, SIAM J. Control Optim., 41, (2003), 1843--1867.
\bibitem{HLE} J. Henriques, J. Lemos, J. E\c{c}a, L. Gato, A. Falc\~ao,  A high-order discontinuous Galerkin method with mesh refinement for optimal control, Automatica, 85, (2017), 70--82. 
\bibitem{NV12} I. Neitzel and B. Vexler, A priori error estimates for space-time finite element discretization of semilinear parabolic optimal control problems, Numer. Math., 120, (2012), 345--386.
\bibitem{Osm05} N. P. Osmolovskii and H. Maurer, Equivalence of second order optimality conditions for bang-bang control problems. Part 1: main results, Control Cybern., 34, (2005), 927--950.
\bibitem{Osm07} N. P. Osmolovskii and H. Maurer, Equivalence of second order optimality conditions for bang-bang control problems. Part 2: proofs, variational derivatives and representations, Control Cybern., 36, (2007), 5--45.
\bibitem{RK} I. M. Ross, M. Karpenko, A review of pseudospectral optimal control: From theory to flight. Annual Reviews in Control, 36, (2012), 182--197.
\bibitem{SS00} D. Sch\"otzau and C. Schwab, An {\it hp} a priori error analysis of the DG time-stepping method for initial value problems, Calcolo, 37, (2000), 207--232.
\bibitem{VV88} J. Vlassenbroeck and R. Van Dooren, A Chebyshev technique for solving nonlinear optimal control problems, IEEE T. Automat. Contr., 33, (1988), 333--349.
\end{thebibliography}
\end{document}